\documentclass[11pt,a4paper,oneside]{article}
\usepackage{CJJT-survey}
\newcommand{\ignore}[1]{}
\normalbaselines
\usepackage{appendix}
\usepackage{mathtools}
\usepackage{mathcomp}
\usepackage{textcomp}
\usepackage[auth-lg]{authblk}
\usepackage{biblatex}
\usepackage{caption}
\usepackage{subcaption}
\usepackage{bm}
\usepackage{multicol}
\usepackage{rotating}
\usepackage{hyperref}
\usepackage{url}
\usepackage{array}

\usepackage{kbordermatrix}

\newenvironment{sproof}{\begin{proof}[Sketch proof.]}{\end{proof}}

\DeclareMathOperator{\rank}{rank}

\DeclareMathOperator{\cl}{cl}

\DeclareMathOperator{\lk}{lk}
\newcommand{\val}{{\rm val}}

\newcommand{\MS}{{\mathcal S}}

\newcommand{\MF}{{\mathcal F}}

\newcommand{\MC}{{\mathcal C}}

\newcommand{\MX}{{\mathcal X}}
\newcommand{\MY}{{\mathcal Y}}

\newcommand{\MM}{{\mathcal M}}
\newcommand{\MR}{{\mathcal R}}
\newcommand{\ZZ}{\Z} 
\newcommand{\RR}{{\R}}

\newcommand{\sm}{{\setminus}}
\newcommand{\DD}{{\vec{G}}}

\usepackage{xcolor}
\usepackage{graphicx} 
\usepackage{booktabs} 
\usepackage{pgf,tikz}
\usetikzlibrary{positioning}
\usepackage{comment}

\title{Rigidity of Graphs and  Frameworks:\\ A Matroid Theoretic Approach}
\author{James Cruickshank, Bill Jackson, Tibor Jord\'an, Shin-ichi Tanigawa}
\date{}
\bccaddressa{School of Mathematical and Statistical Sciences,\\ University of Galway,\\ Ireland,\\ H91 TK33.}
\bccaddressb{School of Mathematical Sciences,\\ Queen Mary University of London,\\ United Kingdom,\\E1 4NS.}
\bccaddressc{Department of Operations Research,\\ E\"otv\"os Lor\'and University,\\Hungary,\\1117}
\bccaddressd{Department of Mathematical Informatics,\\ University of Tokyo,\\Japan,\\113-8656}
\bccemaila{james.cruickshank@universityofgalway.ie}
\bccemailb{b.jackson@qmul.ac.uk}
\bccemailc{tibor.jordan@ttk.elte.hu}
\bccemaild{tanigawa@mist.i.u-tokyo.ac.jp}

\begin{document}

\maketitle
\begin{abstract}
    A $d$-dimensional (bar-and-joint) framework $(G,p)$ consists of a graph $G=(V,E)$ and a realisation $p:V\to {\R}^d$. It is rigid if every continuous motion of the vertices which preserves the lengths of the edges is induced by an isometry of ${\R}^d$. The study of rigid frameworks has increased rapidly since the 1970’s stimulated by numerous applications in areas such as civil and mechanical engineering, CAD, molecular conformation, sensor network localisation and low rank matrix completion. We will describe some of the main results in combinatorial rigidity theory and their applications to other areas of combinatorics, putting an emphasis on links to matroid theory.
\end{abstract}

\section{Introduction}
The study of the rigidity of frameworks can be traced back to a claim made by Euler in 1776    that ``A closed spatial figure allows not changes, as long as it is not ripped apart." Giving a formal proof for Euler's claim turned out to be a challenging problem. A celebrated result of  Cauchy in 1813 implies that every triangulated convex spherical surface  in $\R^3$ is rigid.
Gluck~\cite{G} showed  that every generic triangulated spherical surface in $\R^3$ is rigid.
Rather surprisingly,  Connelly \cite{C77,C82} disproved Euler's original claim by constructing a flexible triangulated spherical surface.

We can use the language of rigidity theory to give a more precise  description of these results.
A $d$-dimensional {\em (bar-joint) framework} is defined as a pair  $(G,p)$  consisting of a graph $G$ and a {\em realisation} $p:V(G)\rightarrow \R^d$.
We will also say that $(G,p)$ is a {\em realisation} of $G$ in $\R^d$.
Two realisations $(G, p)$ and $(G, q)$ 
of $G$ 
in $\R^d$ 
are {\em congruent} 
if $(G, p)$ can be obtained from $(G, q)$ by an isometry of $\R^d$, i.e., a combination of translations, rotations
and reflections. The framework $(G, p)$ is {\em globally rigid} if every framework which has the
same edge lengths as $(G, p)$ is congruent to $(G, p)$. It is {\em rigid} if every continuous motion of
the vertices of $(G, p)$ in $\R^d$ which preserves the lengths of its edges results in a framework which
is congruent to $(G, p)$, and otherwise it is {\em flexible}. 
These concepts are illustrated in  Figure \ref{fig1}.

\begin{figure}[t]
\small
\begin{center}
\unitlength .8mm 
\linethickness{0.4pt}
\ifx\plotpoint\undefined\newsavebox{\plotpoint}\fi 
\begin{picture}(161,32.25)(100,35)
\put(122,41){\line(0,1){20}}
\put(208,41.25){\line(0,1){20}}
\put(122,61){\line(1,0){10}}
\put(208,61.25){\line(1,0){10}}
\put(132,61){\line(0,-1){20}}
\put(218,61.25){\line(0,-1){20}}
\put(132,41){\line(-1,0){10}}
\put(218,41.25){\line(-1,0){10}}
\multiput(162,41)(.042016807,.071428571){238}{\line(0,1){.071428571}}
\put(172,58){\line(1,0){10}}
\multiput(182,58)(-.042016807,-.071428571){238}{\line(0,-1){.071428571}}
\put(172,41){\line(-1,0){10}}
\put(122,41){\circle*{2}}
\put(208,41.25){\circle*{2}}
\put(122,61){\circle*{2}}
\put(208,61.25){\circle*{2}}
\put(132,61){\circle*{2}}
\put(218,61.25){\circle*{2}}
\put(132,41){\circle*{2}}
\put(218,41.25){\circle*{2}}
\put(172,41){\circle*{2}}
\put(182,58){\circle*{2}}
\put(162,41){\circle*{2}}
\put(172,58){\circle*{2}}
\put(119,41){\makebox(0,0)[cc]{$v_1$}}
\put(205,41.25){\makebox(0,0)[cc]{$v_1$}}
\put(119,61){\makebox(0,0)[cc]{$v_2$}}
\put(205,61.25){\makebox(0,0)[cc]{$v_2$}}
\put(135,61){\makebox(0,0)[cc]{$v_3$}}
\put(221,61.25){\makebox(0,0)[cc]{$v_3$}}
\put(135,41){\makebox(0,0)[cc]{$v_4$}}
\put(221,41.25){\makebox(0,0)[cc]{$v_4$}}
\put(159,41){\makebox(0,0)[cc]{$v_1$}}
\put(175,41){\makebox(0,0)[cc]{$v_4$}}
\put(169,58){\makebox(0,0)[cc]{$v_2$}}
\put(185,58){\makebox(0,0)[cc]{$v_3$}}
\put(126.25,35.75){\makebox(0,0)[cc]{$(G,p_0)$}}
\put(213.5,34.75){\makebox(0,0)[cc]{$(G+e,p_0)$}}
\put(171.25,35.75){\makebox(0,0)[cc]{$(G,p_1)$}}
\multiput(207.75,41)(.0421686747,.0833333333){249}{\line(0,1){.0833333333}}
\put(142.25,50){\vector(1,0){13.5}}
\end{picture}
\caption{The 2-dimensional frameworks $(G,p_0)$ and $(G,p_1)$ are not rigid since $(G,p_1)$ can be obtained from  $(G,p_0)$ by a
continuous motion of its vertices in $\R^2$ which preserves all edge lengths, but
changes the distance between $v_1$ and $v_3$. 
The 2-dimensional framework $(G+e,p_0)$ 
is rigid. It is not globally rigid, however, since we can obtain a non-congruent framework with the same edge lengths by reflecting $v_4$ in the line through $v_1v_3$.
} \label{fig1}
\end{center}
\end{figure}
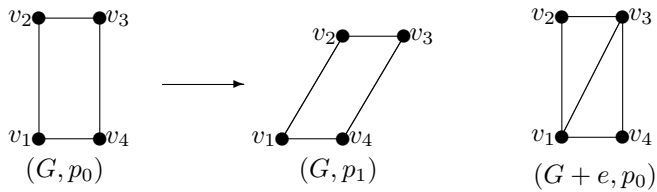

It is straightforward to show that a realisation of a graph on the real line is rigid if and only if the graph is connected: necessity is obvious and sufficiency follows by an easy induction on the number of vertices. In contrast, Abbot~\cite{Abbot08} showed that  it is NP-hard to determine whether a given
$d$-dimensional framework is rigid when $d\geq 2$, and Saxe~\cite{Saxe79} showed the same holds for  global rigidity when $d\geq 1$.
These problems become  more tractable, however, if we restrict attention
  to {\em generic frameworks} i.e.~frameworks $(G,p)$ in which the multiset of coordinates of all
points $p_v$, $v\in V$, is algebraically independent over $\mathbb Q$. 
Gluck~\cite{G} showed that the rigidity of a generic framework 
depends only on its underlying graph.
This allows us to define
a graph $G$ as being {\em rigid in $\R^d$} if some (or equivalently every) generic realisation of $G$
in $\R^d$ is rigid.
Analogous, but much deeper, results of Connelly \cite{C} and Gortler, Healy and Thurston \cite{GHT} 
imply that the global rigidity of a generic framework depends only on its underlying graph so we can define a graph
$G$ as being {\em globally rigid} in $\R^d$ if some (or equivalently every) generic realisation of $G$ in $\R^d$ is globally rigid.
The problem of finding a combinatorial characterisation of rigid or globally rigid graphs  in $\mathbb R^d$ has been solved when $d=1,2$ but is a major open problem in discrete geometry for $d \geq 3$.

Combinatorial rigidity theory has expanded rapidly since the 1970's. It has branched out  to include several other types of frameworks and has found applications in many diverse areas. The aim of this survey is to introduce the reader to some foundational results, proof techniques, and open problems in rigidity and to describe some striking applications to other areas of combinatorics. We refer the reader wishing to discover more to the notes \cite{Jnotes,Wchapter}, text books \cite{GSS93,CG-book} and survey articles on rigidity in \cite{encyc}.
We will assume a knowledge of basic concepts in matroid theory and refer the reader to \cite{Ox} as a source text for these.

\section{Terminology and preliminary  results}

\subsection{Graphs and matroids}\label{sec:g+m}
We will use the term {\em multigraph} for graphs which may contain loops or multiple edges and will reserve the term {\em graph} for graphs without loops and multiple edges. Given a finite set $V$, we use $K_V$ to denote  the complete graph with vertex set $V$ and $\binom{V}{2}$ to denote its edge set.  We say that a multigraph $G=(V,E)$  is {\em $(s,t)$-sparse} for two integers $s,t$ if $|F|\leq s|V(F)|-t$ for all $\emptyset \neq F\subseteq E$ and  {\em $(s,t)$-tight} if it also satisfies $|E|=s|V|-t$. Given $X\subseteq V$ we use $i_G(X)$ to denote the number of edges in $G[X]$, the subgraph of $G$ induced by $X$. Thus $G$ is $(s,t)$-sparse if and only if $i(X)\leq s|X|-t$ for all $X\subseteq V$ which induce at least one edge in $G$. We will use the notation $G/uv$ to denote the graph obtained from $G$ by contracting an edge $e=uv$.

The following result of Edmonds \cite{E} implies that the edge sets  of the $(s,t)$-sparse subgraphs of $K_V$ are the independent sets of a matroid on $\binom{V}{2}$.

\begin{theorem}
\label{thm:edmonds}
Let  $E$ be a finite set and $f:2^E\rightarrow \mathbb{Z}$ be a non-decreasing, submodular function. Put
\[
{\cal I}_f:=\{F\subseteq E: |I|\leq f(I) \text{ for any $I\subseteq F$ with $I\neq \emptyset$} \}.
\]
Then $\MM_f:=(E, {\cal I}_f)$ is a matroid with rank function $\hat f:2^E\rightarrow \mathbb{Z}$ given by
$$\hat{f}(F):=\min\left\{|F_0|+\sum_{i=1}^k f(F_i) : \text{$F_0\subseteq F$ and  $\{F_1,\dots, F_k\}$ is a partition  of $F\sm F_0$}\right\}.
$$
\end{theorem}

We refer to the matroid $\MM_f$ given by Edmonds' theorem as the {\em matroid induced by $f$}. For any finite set $V$ and integers $s,t$, the {\em $(s,t)$-sparsity matroid of $K_V$} is the matroid on $\binom{V}{2}$ induced by the function  $f:F\mapsto s|V(F)|-t$ for all $F\subseteq \binom{V}{2}$.

Nash-Williams \cite{NW} showed that  the $(k,k)$-sparsity matroid of $K_V$ is the matroid union of $k$ copies of the cycle matroid of $K_V$ for any integer $k\geq 1$. This implies,  in particular,  that it can be realised as the row matroid  of a  $|\binom{V}{2}|\times k|V|$-matrix $A$ defined as follows. We first choose a generic map $q:\binom{V}{2}\to \R^k$ 
and a reference ordering of the vertices in $V$. We then define the 
row of $A$ indexed by the edge $e=uv$ with $u<v$ 
to be
\[
\kbordermatrix{
  & &  u & & v & \\
 e=uv & 0 \dots 0 & q_e & 0\dots 0 & -q_e & 0\dots 0
}.
\]

\subsection{Infinitesimal rigidity and rigidity matroids}

A basic step in deciding whether  a given  $d$-dimensional realisation of a graph $G=(V,E)$ is rigid is to first linearise the problem.  We consider the {\em measurement map} $f_G:\R^{d|V|}\to \R^{|E|}$  which maps each $p\in \R^{d|V|}$ onto the vector of squared edge lengths of the framework $(G,p)$. The {\em rigidity matrix} $R(G,p)$ of the framework is the Jacobian matrix  of $f_G$ evaluated at $p$.
It is the $|E|\times d|V(G)|$ matrix in which each row is indexed by an edge, sets of $d$ consecutive columns are indexed by the vertices, and the row indexed by the  edge $e=uv$ has the form:
\[
\kbordermatrix{
  & &  u & & v & \\
 e=uv & 0 \dots 0 & p_u-p_v & 0\dots 0 & p_v-p_u & 0\dots 0
}.
\]
The space of {\em infinitesimal motions}  of $(G,p)$ is given by the right kernel of $R(G,p)$. It consists of the vectors $\dot p:V\to \R^d$ with the property that $\dot p_u-\dot p_v$ is orthogonal to $p_u-p_v$ for all $uv\in E$. Maxwell pointed out in his fundamental work on statics \cite{Max} that the space of infinitesimal motions will have a subspace of {\em trivial motions} of dimension $\binom{d+1}{2}$, generated by the infinitesimal translations and rotations of $\R^d$, whenever $p(V)$ affinely spans $\R^d$. He deduced that 
\begin{equation}\label{eq:max}
\mbox{$\rank R(G,p)\leq d|V|-\binom{d+1}{2}$
}
\end{equation}
when $|V|\geq d+1$, and that $(G,p)$ will be rigid whenever equality holds (since this will imply that every infinitesimal motion of  $(G,p)$ is an infinitesimal isometry of $\R^d$).  We say that $(G,p)$ is {\em infinitesimally rigid} if $|V|\geq d+1$ and equality holds in  \eqref{eq:max}, or $|V|\leq d$ and $\rank R(G,p)=\binom{|V|}{2}$. The 2-dimensional framework on the right of Figure \ref{fig:stress} is an example of a rigid framework which is not infinitesimally rigid (since it has a non-trivial infinitesimal motion which fixes each vertex of degree three and gives each vertex of degree two an arbitrary `infinitesimal velocity' in the direction perpendicular to the line through its neighbours).

Asimow and Roth \cite{AR} used elementary differential geometry to show that infinitesimal rigidity is equivalent to rigidity whenever $p$ is a regular point of the measurement map. Since the entries in $R(G,p)$ are polynomial functions of the coordinates of $p$ with integer coefficients, this gives

\begin{theorem}\label{thm:AR}
 A generic framework is rigid in $\R^d$ if and only if it is infinitesimally rigid in $\R^d$.   
\end{theorem}

Equivalently, a graph  is (generically) rigid in $\R^d$ if and only if it has an infinitesimally rigid realisation in $\R^d$. These observations have important algorithmic consequences: they tell us that the problem of deciding whether a graph  is rigid in $\R^d$ belongs to NP and, in combination with the Schwarz-Zippel Lemma \cite{S,Z},  give rise to a randomised polynomial algorithm for solving this problem. Unfortunately, this tells us nothing about the structure of graphs which are rigid in $\R^d$. We will use techniques from graph theory,  matroid theory and geometry to obtain structural information on such graphs. 

The {\em rigidity matroid} $\MR(G,p)$ of a framework $(G,p)$ is the row matroid of the rigidity matrix $R(G,p)$. Thus,  $\MR(G,p)$ is a matroid on $E$ in which a set of edges $F\subseteq E$ is independent if the rows of $R(G,p)$ indexed by $F$ are linearly independent. Since the rows of $R(G,p)$ are nonzero scalar multiples of the rows of an oriented edge/vertex incidence matrix of $G$ when $p$ is an injection from $V$ to $\R$, the rigidity matroid of any such 1-dimensional  realisation of $G$ is just the cycle matroid of $G$.

The  {\em $d$-dimensional rigidity matroid} $\MR_d(G)$ of a graph $G$ is defined to be the rigidity matroid $\MR(G,p)$  of any generic realisation of $G$ in $\R^d$. Thus, $G$ is rigid in $\R^d$ if and only if the rank of $\MR_d(G)$ is
$d|V|-\binom{d+1}{2}$ when $|V|\geq d+1$. Since every independent set in a matroid can be extended to a base, an algorithm to decide whether any given set of edges of $G$ is independent in $\MR_d(G)$ would allow to determine whether  $G$ is rigid in $\R^d$.

We will simplify terminology by describing subgraphs $H$ of $G$   using properties of their edge set  in $\MR_d(G)$. For example we will say $H$ is $\MR_d$-independent to mean $E(H)$ is independent in $\MR_d(G)$ and will use $r_d(H)$ to denote the rank of $E(H)$ in $\MR_d(G)$. We can apply  \eqref{eq:max} to the induced subgraphs of $G$ to obtain the following graph theoretic necessary condition for $\MR_d$-independence.

\begin{lemma}[Maxwell's independence criterion] \label{lem:max}
Suppose $G$ is an $\MR_d$-independent graph. Then $i_G(X)\leq d|X|-\binom{d+1}{2}$ for all $X\subseteq V$ with $|X|\geq d+1$.
\end{lemma}

Maxwell's criterion is also sufficient to imply  $\MR_d$-independence when $d=1$ since $\MR_1(G)$ is the cycle matroid of $G$ and Maxwell's criterion is equivalent to saying that $G$ is a forest. We will see in the following sections that it remains sufficient when $d=2$ (and hence the 2-dimensional rigidity matroid is equal to the $(2,3)$-sparsity matroid), but is not sufficient to imply $\MR_d$-independence when $d\geq 3$. 

\subsection{Global rigidity and stress matrices}
\label{sec:GRandStressMatrices}

The space of {\em equilibrium stresses}  of a $d$-dimensional framework $(G,p)$ is given by the left kernel of $R(G,p)$. It consists of the vectors $\omega:E\to \R$ with the property that, for all $v\in V$,  $\sum_{uv\in E}\omega_{uv} (p_u- p_v)={\bf 0}$, see Figure \ref{fig:stress}. Given an equilibrium stress $\omega$ of $(G,p)$ we can define the corresponding stress matrix $\Omega$ as the $|V|\times |V|$ symmetric matrix in which the entry in row $u$ and column $v$ is $-\omega_{uv}$ if $u\neq v$ and $uv\in E$, $0$ if $u\neq v$ and $uv\notin E$, and the diagonal entries are chosen so that the row and column sums of $\Omega$ are each equal to zero. Connelly showed  in his fundamental work on global rigidity \cite{C82a,C} that the rank of $\Omega$ is at most $|V|-d-1$ and that $(G,p)$ is globally rigid in the case when equality holds and $p$ is generic. Gortler, Healy, and Thurston \cite{GHT} subsequently verified a conjecture of Connelly by showing that the converse of the second part of the previous statement also holds (when $|V|\geq d+2$). This gives us

\begin{theorem}\label{thm:stress} A generic $d$-dimensional framework $(G,p)$  is globally rigid if and only if $G$ is a complete graph on at most $d+1$ vertices or $(G,p)$ has a stress matrix of rank $|V|-d-1$.
\end{theorem}

Theorem \ref{thm:stress} plays an analogous role for generic global rigidity as Theorem \ref{thm:AR} does for generic rigidity. In particular it implies that the global rigidity of a generic $d$-dimensional  framework $(G,p)$ depends only on the graph $G$, and that $G$ is (generically) globally rigid in $\R^d$ if and only if it has a  realisation in $\R^d$ which is infinitesimally rigid and has a stress matrix of rank $|V|-d-1$. This, in turn, gives rise to a randomised polynomial algorithm for deciding whether $G$ is (generically) globally rigid in $\R^d$, see \cite{GHT, CW10}. An example of a graph $G=(V,E)$ which is not generically globally rigid in $\R^2$ but still has a realisation with a stress matrix of  rank $|V|-3$ is given on the right of Figure \ref{fig:stress}.

\begin{figure}[t]
\begin{center}
\tiny
\input{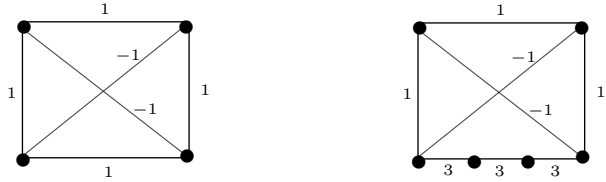}
\caption{Two 2-dimensional  stressed frameworks with an associated stress matrix of maximum possible rank. The framework on the right is obtained from the one on the left by  subdivisions. This operation preserves the maximum rank property of a stressed framework. The framework on the right is not infinitesimally rigid and its underlying graph does  not have enough edges to be generically (globally) rigid in $\R^2$.
} \label{fig:stress}
\end{center}
\end{figure}

Hendrickson \cite{H} gave the following graph theoretic necessary conditions for a graph to be generically globally rigid. We say that a graph $G$ is {\em redundantly rigid in $\R^d$} if $G-e$ is rigid in $\R^d$ for all edges $e$ of $G$.

\begin{theorem}\label{thm:hend} Suppose $G$ is a graph which is globally rigid in $\R^d$. Then either $G$ is a complete graph on at most $d+1$ vertices or $G$ is $(d+1)$-connected and redundantly rigid in $\R^d$.
\end{theorem}

Necessity of the connectivity condition follows since, if $G$ has a vertex-cut of size less than $d+1$, then, for any generic realisation, we can reflect one side of the cut in a hyperplane containing the vertices of the cut. This  operation will preserve the length of all the edges but will change the distance between any  two vertices on opposite sides of the cut. Hendrickson 
verifies necessity of redundant rigidity by showing that, if $G$ has at least $d+2$ vertices and has  an edge $e=uv$ for which $G-e$ is not rigid, then, for any generic realisation $(G,p)$,
we can continuously deform $(G-e,p)$ in such a way that we preserve all its edge lengths and eventually arrive at a non-congruent framework $(G-e,p')$  which satisfies $\|p(u)-p(v)\|=\|p'(u)-p'(v)\|$.

We will see that Hendrickson's necessary conditions for global rigidity are also sufficient when $d=1,2$ but not for $d\geq 3$. Another necessary condition for global rigidity in $\R^d$ has recently been obtained by Garamv\"olgyi, Gortler and Jord\'an \cite{GGJ}. 

\begin{theorem}
\label{theorem:mconnected}
Let $G$ be a globally rigid graph in $\RR^d$ on at least $d+2$ vertices. 
Then $\MR_d(G)$ is a connected matroid.
\end{theorem}

This necessary condition is implied by Hendrickson's conditions when $d=1,2$ but not for $d\geq 3$. The combination of all three conditions is still not sufficient to imply global rigidity in $\R^d$ when $d\geq 3$.

The following sufficient condition for global rigidity is given in \cite{ST}. It can be viewed as a partial converse to Theorem \ref{thm:hend}.

\begin{theorem}
If $G=(V,E)$ is a graph and $G-v$ is  rigid
in $\R^d$ for all $v\in V$, then $G$ is globally rigid in $\R^d$.
\label{shin2rigid}
\end{theorem}

\subsection{Inductive constructions}\label{sec:induct}
Proofs that a given family of graphs  are rigid, or globally rigid,  in $\R^d$ are often inductive in nature and proceed by using a suitably chosen graph operation to reduce to a smaller graph in the family, applying induction to deduce that the smaller graph is rigid,  and then using the property that the inverse of our chosen operation preserves rigidity. We will give examples of such `rigidity preserving operations' 
in this subsection. The first two operations were introduced by Henneberg in \cite{henn}.

\subsubsection{0-extensions}
Given a graph $H$, the $d$-dimensional {\em $0$-extension} operation creates a new graph $G$ by adding a vertex $v$ of degree at most $d$ to $H$. It is well-known that the $d$-dimensional 0-extension operation preserves 
$\MR_d$-independence.
This follows from the following more general statement for the rigidity matrix of an  arbitrary framework. It can be easily verified using the fact that the only non-zero entries in the columns of this matrix labelled by $v$ occur in the $d_G(v)$ rows labelled by the edges incident to $v$.
\begin{lemma}\label{lem:0extension}
Suppose  $(G,p)$ is a framework in $\R^d$ and $v$ is a vertex in $G$.
If 
$\{p_v-p_u: u\in N_G(v)\}$ is linearly independent in $\R^d$
then $\rank R(G,p)=\rank  R(G-v, p|_{V-v})+d_G(v)$. 
\end{lemma}

\subsubsection{1-extensions}
Given a graph $H$, the $d$-dimensional {\em $1$-extension} operation creates a new graph $G$  from $H$ by deleting an edge $xy$ and then  adding a vertex $v$ and $d+1$ new edges from $v$ to $x$, $y$ and $d-1$ other vertices of $H$, see Figure \ref{fig:1ext}.

\begin{figure}[t]
\small
\begin{center}
\input{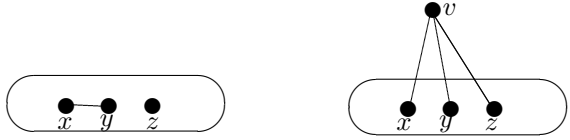}
\end{center}
\caption{The 2-dimensional $1$-extension operation.}
\label{fig:1ext}
\end{figure}

\begin{lemma}\label{lem:1extension}
The $d$-dimensional 1-extension operation preserves both  rigidity and global rigidity  in $\R^d$.
\end{lemma}
\begin{sproof} Rigidity can be verified by using a `special position' argument. Suppose $H$ is rigid in $\R^d$, $G$ is a 1-extension of $H$ and $H=G-v+xy$.
Choose a generic framework $(H,p)$ in $\R^d$ and extend it to a framework $(G,p')$ by placing $v$ on the line through $p_x,p_y$. Then $(H,p)$ is infinitesimally  rigid in $\R^d$ and we can use Lemma \ref{lem:0extension} to deduce that $(G+xy,p')$ is also infinitesimally  rigid in $\R^d$. On the other hand, the fact that the points $p'_v,p'_x,p'_y$ are collinear implies that the rows of $R(G+xy,p')$ indexed by $xy,vx,vy$ are minimally linearly dependent and hence $R(G+xy,p')$ and $R(G,p')$ have the same rank. This gives us an infinitesimally rigid realisation of $G$ and implies that it is rigid in $\R^d$. 

Connelly \cite{C} applied a similar  special position argument to show that the 1-extension operation preserves  the property of having a `maximum rank stress matrix' and hence preserves global rigidity in $\R^d$ by Theorem \ref{thm:stress}. We refer the reader to \cite{C} for more details.
\end{sproof}

\subsubsection{Vertex splitting} 

Suppose $z$ is a vertex of a graph $H$. 
The {\em ($d$-dimensional) vertex splitting operation} constructs a new graph $G$ from $H$  by deleting $z$ and then adding two new vertices $u$ and $v$ with $N_G(u)\cup N_G(v)=N_H(z)\cup \{u,v\}$ and $|N_G(u)\cap N_G(v)|\geq d-1$, see Figure \ref{fig:vsplit}. (We can view vertex splitting as an inverse operation to contracting the edge $uv$.) 
Whiteley \cite{Wsplit} showed that 
vertex splitting preserves generic rigidity in $\R^d$.

\begin{figure}[t]
\small
\begin{center}
\input{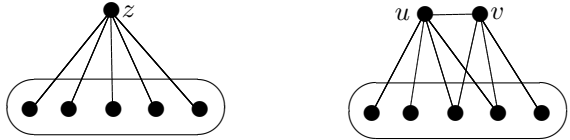}
\end{center}
\caption{The 3-dimensional vertex splitting operation.}
\label{fig:vsplit}
\end{figure}

\begin{lemma}\label{lem:split} 
Suppose $G$ is a graph, $uv\in E$ and  $|N_G(u)\cap N_G(v)|\geq d-1$.
If $G/uv$ is rigid in $\R^d$ then $G$ is rigid in $\R^d$.
\end{lemma}
\begin{sproof} Whiteley's proof uses yet another special position argument. He chooses a generic framework $(G/uv,p)$ and extends it to a framework $(G,p')$ by putting $p'_u=p'_v=p_z$, where $z$ is the vertex created by contracting $uv$. He then applies a limiting argument to the rigidity matrix to deduce that it has full rank for all realisations sufficiently close to this special position. We refer the reader to \cite{Wsplit} for more details.
\end{sproof}

Whiteley conjectures that an analogous result holds for global rigidity in $\R^d$ as long as the vertex splitting is {\em non-trivial} i.e.~the degrees of $u$ and $v$ in $G$ are at least $d+1$. Connelly shows in his online notes \cite{Cnotes} that vertex splitting preserves global rigidity whenever $G$ has an infinitesimally rigid realisation in $\R^d$ in which $u$ and $v$ are coincident. A more precise statement and proof of his result is given in \cite{J}. Jord\'an and Tanigawa  verify another special case of Whiteley's conjecture in \cite{JT-braced} by showing that a graph is globally rigid in $\R^d$ whenever it can be constructed from a globally rigid graph of maximum degree at most $d+2$ by a sequence of non-trivial vertex splits.

\subsubsection{Gluing and substitution}
Our next two results describe situations in which we can construct a graph which is (globally) rigid in $\R^d$ by `gluing together' two smaller  graphs or replacing one subgraph by another subgraph. 

\begin{lemma}\label{lem:glue} Let $G$ be a graph and $G_1,G_2$ be subgraphs of $G$ with  $V(G_1)\cap V(G_2)=X$, $G=G_1\cup G_2$ and $V(G_1)\sm V(G_2)\neq \emptyset\neq V(G_2)\sm V(G_1)$.\\
(a) Suppose $G_1,G_2$ are rigid in $\R^d$. Then $G$ is rigid in $\R^d$ if and only if $|X|\geq d$.\\
(b) Suppose $G_1,G_2$ are globally rigid in $\R^d$. Then $G$ is globally rigid in $\R^d$ if and only if $|X|\geq d+1$.
\end{lemma}
\begin{sproof} Necessity is obvious since, for any generic realisation $(G,p)$ in $\R^d$, we can rotate a component of $G-X$ about a $(d-2)$-dimensional affine subspace containing $p(X)$ if $|X|<d$, and reflect it in a $(d-1)$-dimensional affine subspace containing $p(X)$ if $|X|=d$. We refer the reader to \cite{Wchapter} for  proofs of sufficiency.  
\end{sproof}

\begin{lemma}\label{lem:sub} Let $G_i=(V_i,E_i)$ be a graph for $i=1,2,3$, with $V_1\cap V_2=X=V_1\cap V_3$.\\
(a) If $G_1\cup G_2$ and $G_3$ are both rigid in $\R^d$ and $|X|\geq d$, then $G_1\cup G_3$ is rigid in $\R^d$.\\
(b) If $G_1\cup G_2$ and $G_3$ are both globally rigid in $\R^d$ and $|X|\geq d+1$, then $G_1\cup G_3$ is globally rigid in $\R^d$.
\end{lemma}
\begin{sproof} (a) Choose generic realisations $p$ of $G_1\cup G_2$ and $q$ of $G_1\cup G_3$ in $\R^d$ such that $p|_{G_1}=q|_{G_1}$. Let $\dot q$ be an infinitesimal motion of $(G_1\cup G_3,q)$. Since $G_3$ is rigid in $\R^d$ we may assume (by composing $\dot q$ with a suitable trivial infinitesimal motion) that $\dot q|_{G_3}=0$. Then $\dot q$ corresponds to an infinitesimal motion $\dot p$ of $(G_1\cup G_2,p)$ with $\dot p|_{G_1}=\dot q|_{G_1}$ and $\dot p|_{G_2}=0$. Since $G_1\cup G_2$ is rigid and $|X|\geq d$ this gives $\dot p=0$. Hence $\dot q=0$ and $G_1\cup G_3$ is rigid.

Part (b) can be proved similarly.
\end{sproof}

\subsection{Linked pairs of vertices} 

Let $(G,p)$ be a $d$-dimensional framework. We say a pair of  vertices $\{u,v\}$ of $G$ are {\em linked} in $(G,p)$ if every continuous motion of its vertices in $\R^d$ which preserves the lengths of its edges, preserves the distance between $u$ and $v$. Thus $(G,p)$ is rigid if and only if all its pairs of vertices are linked. In the case when $p$ is generic, $\{u,v\}$ is linked in $(G,p)$ if and only $r_d(G+uv)=r_d(G)$; this follows from the fact that every infinitesimal motion of a generic framework can be extended to a finite motion. Hence '$\{u,v\}$-linkedness' is a generic property and we can define a pair of  vertices $\{u,v\}$ of $G$ to be {\em linked (in $G$) in $\R^d$} if they are linked in every (or equivalently some) generic realisation of $G$ in $\R^d$.

\begin{figure}[t]
\begin{center}
\includegraphics[scale=1.1]{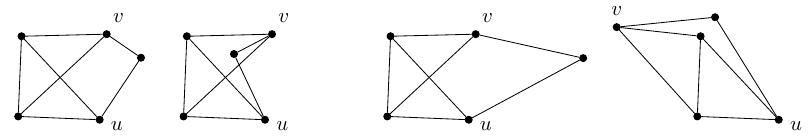}

\vspace{-0.5cm}
\end{center}
\caption{Two pairs of generic realisations of a graph $G$ in $\R^2$. The vertex pair $\{u,v\}$ is globally linked in the two frameworks on the left since these are the only two non-congruent realisations of $G$ in $\R^2$ with these edge lengths and the distance between $u$ and $v$ is the same in both of them. On the other hand, $\{u,v\}$ is not globally linked in the two frameworks on the right since they have the same edge lengths and the the distance between $u$ and $v$ is different. Thus $\{u,v\}$ is weakly globally linked, but not globally linked in $G$ in $\R^2$.}
\label{fig:gen1}
\end{figure}

Similarly, we say that $\{u,v\}$ is {\em globally linked} in $(G,p)$ if, for every framework $(G,q)$ in $\R^d$ which has the same edge lengths as $(G,p)$, we have $\|q(u)-q(v)|=\|p(u)-p(v)\|$.  It is not so  straightforward to define the generic version of this concept since $\{u,v\}$ can be globally linked in one generic realisation of $G$ in $\R^d$ and not in another, see Figure \ref{fig:gen1}.
Hence we say $\{u,v\}$ is {\em globally linked (in $G$) in $\R^d$}
if $\{u,v\}$ is globally inked in $(G,p)$ for {\em all} generic $p:V\to \R^d$, and 
$\{u,v\}$ is {\em weakly globally linked (in $G$) in $\R^d$}
if $\{u,v\}$ is globally inked in $(G,p)$ for {\em  some} generic $p:V\to \R^d$. It follows immediately that $G$ is globally rigid in $\R^d$ if and only every pair of vertices of $G$ is globally linked in $\R^d$. Somewhat surprisingly, Jord\'an and Vill\'anyi \cite{JV} have recently shown that global rigidity can also characterised by weak global linkedness.

\begin{theorem}\label{thm:JV} A graph $G$ is globally rigid in $\R^d$ if and only if every pair of vertices of $G$ is weakly globally linked in $\R^d$.
\end{theorem}

\section{Rigidity and global rigidity in dimensions at most two}

\subsection{Rigidity}
We have already seen that Maxwell's independence criterion characterises rigidity on the line. Pollaczek-Geiringer \cite{PG} showed that it also characterises generic rigidity in the plane in 1935 but her result went largely unnoticed until it was rediscovered by Laman \cite{Lam} in 1970. 
Pollaczek-Geiringer's original paper was brought to the attention of the current day rigidity community by Brigitte Servatius in 2010.

\begin{theorem} \label{thm:PG} A graph with at least two vertices is rigid in $\R^2$ if and only if it has a spanning subgraph $G=(V,E)$ which satisfies $|E|=2|V|-3$ and $i_G(X)\leq 2|X|-3$ for all $X\subseteq V$ with $|X|\geq 2$.
\end{theorem}

Necessity follows from Lemma \ref{lem:max}. To prove sufficiency it will suffice to show that every $(2,3)$-tight graph is rigid in $\R^2$. 
We will sketch two proofs of this statement  to illustrate proof techniques which are commonly used to verify rigidity.

\paragraph{First proof of Theorem \ref{thm:PG}} Let $G=(V,E)$ be a $(2,3)$-tight graph. We show that $G$ is rigid in $\R^2$ using  the inductive strategy outlined in Section \ref{sec:induct}.
Since $G$ is $(2,3)$-tight, $G$ has a vertex $v$ of degree two or three. If $d_G(v)=2$ then $G-v$ is $(2,3)$-tight, so is rigid in $\R^2$ by induction, and we can now apply Lemma \ref{lem:0extension} to deduce that $G$ is rigid in $\R^2$. Hence we may assume that $N_G(v)=\{x,y,z\}$. It will suffice to show that $G-v+uw$ is $(2,3)$-tight for some $u,w\in \{x,y,z\}$ since we can then deduce that $G$ is rigid in $\R^2$ by applying induction and  Lemma \ref{lem:1extension} to $G-v+uw$. 

Suppose for a contradiction that $G-v+uw$ is not $(2,3)$-tight for all $u,w\in \{x,y,z\}$. Let $M$ be the $(2,3)$-sparsity matroid of $K_V$. We have $r_{M}(G)=|E|$ and $r_{M}(G-v)=|E|-3$ since $G$ is $M$-independent. The assumption that $G-v+uw$ is $M$-dependent for all $u,w\in \{x,y,z\}$ now gives $r_{M}(G-v+xy+yz+xz)=r_{M}(G-v)$. Since $E(K_4)$ is a circuit in $M$, this implies that $r_M(G+xy+yz+xz)\leq r_{M}(G-v)+2= r_{M}(G)-1$, and contradicts the fact that  $r_{M}$ is non-decreasing. \qed

\medskip
Our second proof of Theorem \ref{thm:PG} is a version of a proof by Lov\'asz and Yemini \cite{LY} which was given by Whiteley in \cite{WW}.

\paragraph{Second proof of Theorem \ref{thm:PG}} 
Let $G=(V,E)$ be a $(2,3)$-tight graph. We will show that $G$ is rigid in $\R^2$ by constructing an infinitesimally rigid realisation of $G$. Choose distinct vertices $xy\in E$ and let $G^{xy}=(V,E^{xy})$ be the multigraph obtained by adding a new copy of $xy$ to $G$. The hypothesis that $G$ is $(2,3)$-tight implies that $G^{xy}$ is $(2,2)$-tight. Choose a generic map $q:E^{xy}\to \R^2$ and a reference ordering for the vertices in $V$. Let $A^{xy}$ be the $|E^{xy}|\times 2|V|$-matrix in which the entry in the row  indexed by the edge $e=uw$ with $u<w$ and columns indexed by the vertex $v$ is $q(e)$ if $v=u$, $-q(e)$ if $v=w$ and $\bf 0$ if $v\neq u,w$. Then $\rank A^{xy}=|E^{xy}|$ by the discussion in the last paragraph of Section \ref{sec:g+m}. Thus, if $A$ is the matrix obtained by deleting row $xy$ from $A^{xy}$, then we have $\rank A^{xy}=\rank A+1$. This implies that there exists a vector $\alpha^{xy}\in \ker A$ with $\alpha^{xy}_x\neq \alpha^{xy}_y$. By taking a suitable linear combination of the vectors $\alpha^{xy}$, we can construct a vector $\alpha\in \ker A$ with the property  that $\alpha_x\neq \alpha_y$ for all $xy\in E$. We can now define our desired framework $(G,p)$ by putting $p_v=\alpha_v^\perp$ for each $v\in V$, where $(a,b)^\perp:=(-b,a)$. Then each row of $R(G,p)$ will be  a non-zero scalar multiple of the corresponding row of $A$. Hence $\rank R(G,p)=\rank A=2|V|-3$ and $(G,p)$ is infinitesimally rigid. \qed

\medskip
Once we know that the 2-dimensional rigidity matroid is equal to the $(2,3)$-sparsity matroid, Theorem \ref{thm:edmonds} immediately  gives us an expression for the rank function of this matroid.  Lov\'asz and Yemini \cite{LY} showed that this expression can be simplified to
\begin{equation}\label{eq:LY}
r_2(G)=\min_\MX\left\{\sum_{X\in \MX}(2|X|-3)\right\}
\end{equation}
where the minimum is taken over all families $\MX$ of subsets of $V$, each of cardinality  at least two, which cover $E$ and are {\em $1$-thin} i.e. they satisfy $|X_i\cap X_j|\leq 1$ for all distinct $X_i,X_j\in \MX$. In addition they show that the right hand side of \eqref{eq:LY} achieves its minimum value when we take $\MX$ to be the vertex sets of the  {\em $2$-dimensional rigid components of $G$} i.e. the maximal subgraphs of $G$ which are rigid in $\R^2$.

\subsection{Global rigidity}

It is not difficult to see that a graph $G=(V,E)$ is globally rigid on the line if and only if $G$ is a complete graph on at most two vertices or is $2$-connected. Necessity follows from Lemma \ref{lem:glue}.
(Note that the redundant rigidity condition of Theorem \ref{thm:hend} follows from the connectivity condition in this 1-dimensional case.) We can prove  sufficiency by induction on $|E|$. We may assume that $|V|\geq 4$ and that $G$ is minimally 2-connected  i.e. $G-e$ is not 2-connected for all $e\in E$. Choose an ear decomposition of $G$. Since $G$ is minimally 2-connected, the last ear of this ear decomposition is an induced path of length at least two. Let $v$ be an internal vertex of this last ear and $x,y$ be the neighbours of $v$ in $G$. Then $G-v+xy$ is 2-connected and hence is globally rigid on the line by induction. We can now apply Lemma \ref{lem:1extension} to deduce that $G$ is globally rigid on the line. 

 A similar inductive proof technique was used in \cite{JJ} to show that Hendrickson's necessary conditions for global rigidity are also sufficient when $d=2$, based on the generalisation of ear decompositions from graphs to matroids given by Coullard and Hellerstein~\cite{CH}.

\begin{theorem}\label{thm:JJ}
A graph $G$ is globally rigid in $\R^2$ if and only if $G$ is a complete graph on at most $3$ vertices or $G$ is $3$-connected and redundantly rigid in $\R^2$.
\end{theorem}

The key idea in the proof of sufficiency in \cite{JJ} is to 
show that the 2-dimensional rigidity matroid of a $3$-connected, redundantly rigid graph $G=(V,E)$ is connected and then take an ear decomposition of $\MR_2(G)$. 
We can assume inductively that $G-e$ is not both  $3$-connected and redundantly rigid, and also, by a previous result of Berg and Jord\'an \cite{BJ}, that $G$ is not an $\MR_2$-circuit. We next show that  the subgraph of $G$ induced by the last ear in the ear decomposition of $\MR_2(G)$ contains a vertex $v$ of degree 3 such that $\MR_2(G-v+xy$) is connected (and hence $G-v+xy$ is redundantly rigid in $\R^2$) for some $x,y\in N_G(v)$. The argument is rather technical and an additional technical argument is needed to show we can find such a vertex $v$ such that $G-v+xy$ is 3-connected. Once we have such a vertex we can apply induction to $G-v+xy$ and then use Lemma \ref{lem:1extension} to deduce that $G$ is globally rigid. 

A simplified version of this proof appeared in \cite{ST}. A different, and even simpler, proof based on Theorem \ref{thm:JV} has recently been given by Jord\'an and Vill\'anyi in \cite{JV}. In fact, they go further and obtain a complete characterisation of when two vertices in a graph are weakly globally linked in $\R^2$. The related problem of  characterising when two vertices in a graph are globally linked in $\R^2$ is one of the few remaining open problems for generic rigidity in $\R^2$. A conjectured characterisation, together with a proof of sufficiency and other partial results, can be found in  \cite{JJS1, JJS2}.

\section{Rigidity and global rigidity in higher dimensions}

The problem of characterising graphs which are generically rigid or globally rigid in $\R^d$ is open for all $d\geq 3$. We will investigate how the rank formula for the 2-dimensional rigidity matroid given in \eqref{eq:LY} may be extended to higher dimensions. There are many results which show that special families of graphs are rigid in higher dimensions. We refer the reader to the comprehensive survey articles on rigidity in \cite{encyc} for information on these results.

\subsection{Rigidity in 3-space}\label{sec:3-rigidity}

We first note that Maxwell's necessary condition, Lemma \ref{lem:max},  is not sufficient to imply $\MR_3$-independence. The so called {\em double banana graph} of Figure \ref{fig:banana} satisfies Maxwell's condition for $\MR_3$-independence. Since it has $18$  edges, it would be rigid in $\R^3$ if it were $\MR_3$-independent and this would contradict Lemma \ref{lem:glue}(a). 

\begin{figure}[t]
\small
\begin{center}
\input{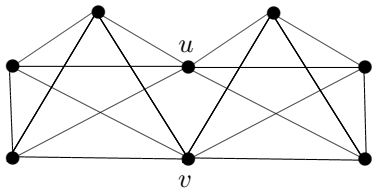}
\end{center}
\caption{The double banana graph.}
\label{fig:banana}
\end{figure}

To motivate a conjectured characterisation of rigidity in $\R^3$, we return to the expression for the rank function of the 2-dimensional rigidity matroid of a graph $G=(V,E)$ given in \eqref{eq:LY} and the result of \cite{LY} that the right hand side of \eqref{eq:LY} achieves its minimum value when we take the 1-thin cover $\MX$ to be the vertex sets of the 2-dimensional rigid components of $G$.  
It is easy to see that $r_2(G)$ is at most $\sum_{X\in \MX}(2|X|-3)$ for any cover $\MX$ of $E$ with vertex sets of size at least  two: this follows from the
fact that any base of $\MR_2(G)$ can contain at most $2|X|-3$ edges in $G[X]$ for each $X\in \MX$.

A similar reasoning tells us that 
\begin{equation}\label{eq:r3}
r_3(G)\leq \min\left\{|F|+\sum_{X\in \MX}(3|X|-6)\right\}
\end{equation}
where the minimum is taken over all $F\subseteq E$ and all covers $\MX$ of $E\sm F$ with sets of size at least three. 
 
Equality does not hold in \eqref{eq:r3} when we  take $G=(V,E)$ to be the double banana graph in Figure \ref{fig:banana} since, for each $X\subseteq V$ with $|X|\geq 3$, we have $i_G(X)\leq 3|X|-6$ and hence the right hand side of \eqref{eq:r3} is at least $|E|=18$. We can, however, obtain a tight upper bound on $r_3(G)$ if we take account of `hinges' in the cover. Consider the graph $G+uv$ and let $G_1$ and $G_2$ be the two copies of $K_5$ in $G+uv$. Let $B$ be a base of $\MR_2(G+uv)$ which contains $uv$ and let $B_i=B\cap E(G_i)$ for $i=1,2$. Then $|B_i|\leq 3|V(G_i)|-6=9$ for each $i=1,2$. Since $uv\in B_1\cap B_2$, this gives $r_3(G)= |B|\leq |B_1|+|B_2|-1= 17$. (We can also show that $r_2(G)\geq 17$ by choosing an arbitrary edge of $G$ and showing that $G-e$ can be constructed recursively from a path of length two by applying the 3-dimensional $0$- and $1$-extension operations. Or equivalently, that $G-e$ can be reduced to a path of length two by their inverse operations.)

We can generalise the argument in the preceding paragraph to obtain a better upper bound on $r_3(G)$ for an arbitrary graph $G=(V,E)$. Given a cover $\MX$  of $E$, we say that $\MX$ is {\em $t$-thin} if any two distinct sets in $\MX$ intersect in at most $t$ vertices, and define the {\em set of hinges of $\MX$} to be 
$$H(\MX)=\left\{\{x,y\}\in \binom{V}{2}:\{x,y\}\subseteq X_1\cap X_2 
\mbox{ for two distinct sets } X_1,X_2\in \MX\right\}.$$
For each hinge $h\in H(\MX\}$, the {\em degree of $h$ in $\MX$, }$d_\MX(h)$, is the number of sets in $\MX$ which contain $h$. 
The {\em $\MR_3$-closure of $G$}, $\cl_3(G)$, is the graph obtained from $G$ by adding an edge between all non-adjacent pairs of vertices which are linked in $\R^3$ (equivalently, adding  all edges $uv$ which satisfy $r_3(G+uv)=r_3(G)$ to $G$). We will refer to the maximal cliques in $\cl_3(G)$ as the {\em rigid clusters} of $G$. Dress et al \cite{D} conjectured that the cover of $E$ given by the rigid clusters of $G$ determines $r_3(G)$:

\begin{conj}
[Dress Conjecture] \label{con:dress} For any graph $G$ 
\begin{equation}\label{eq:dress}
r_3(G)=
|F|+\sum_{X\in \MX}(3|X|-6)
-\sum_{h\in H(\MX)}(d_\MX(h)-1)
\end{equation}
where $F$ is the set of edges of $G$ which belong to rigid clusters of  cardinality two and $\MX$ is the set of rigid clusters 
of cardinality at least three.
\end{conj}

This conjecture is of strong theoretical interest but it would not immediately allow us to calculate $r_3(G)$ since it does not tell us how to determine $\cl_3(G)$. Dress et al offered a second conjecture in \cite{D} which would imply that the problem of determining whether $G$ is rigid in $\R^3$ belongs to NP $\cap$ coNP. They conjectured that $r_3(G)$ is equal to $\min\{|F|+\val(\MX)\} $ where the minimum
is taken over all $F\subseteq E$ and all $2$-thin covers $\MX$ of $E\sm F$ with vertex sets of cardinality  at least three, and 
$$\val(\MX):=\sum_{X\in \MX}(3|X|-6)
-\sum_{h\in H(\MX)}(d_\MX(h)-1)\}.$$ 
Unfortunately this conjecture turned out to be spectacularly false: it was shown in \cite{JJ05} that $\min\{|F|+\val(\MX)\}$ is negative for all graphs with 56 vertices.

We can try to repair the second Dress conjecture by restricting the pairs $(F,\MX)$ over which the minimum is taken to more closely resemble the cover of $E$ 
given by the rigid clusters of $G$. Lemma \ref{lem:glue} implies that the set $\MX$ of rigid clusters  of $G$  is 2-thin. In addition, Theorem \ref{thm:cjt} below gives some evidence that  $\MX$ is {\em $\MR_3$-degenerate} i.e.~$\MX$ 
 can be ordered as $(X_1,X_2,\ldots,X_t)$ in such a way that, for all $2\leq i\leq t$, the set of hinges of the subfamily $\{X_1,X_2,\ldots,X_{i}\}$ which are contained in $X_{i}$ is $\MR_3$-independent (when viewed as a set of edges of $K_V$).

It had already been shown in \cite[Lemma 5 and Conjecture 3.2]{JJ06} that, for any graph $G=(V,E)$,
\begin{equation}\label{eq:degen}
r_3(G)\leq \min\{|F|+\val(\MX)\} 
\end{equation}
where the minimum
is taken over all $F\subseteq E$ and all  $\MR_3$-degenerate covers $\MX$ of $E\sm F$ with sets of cardinality  at least three, and conjectured that equality holds.\footnote{The statement of \cite[Lemma 5]{JJ06} is  in terms of 'iterated covers' and contains the additional hypothesis that the cover is 2-thin. The formulation above is given as \cite[Lemma 6.2]{CJT}.} Theorem \ref{thm:cjt} below suggests that  this conjecture may remain true if we restrict our attention to a very simple type of $\MR_3$-degenerate cover. We say that a vertex cover $\MX$ of $E$ is {\em $k$-shellable} if  $\MX$ 
 can be ordered as $(X_1,X_2,\ldots,X_t)$ in such a way that, for all $2\leq i\leq t$, $|X_i\cap \bigcup_{j=1}^{i-1} X_j|\leq k$. 

\begin{conj}\label{con:cover}
 For any graph $G$ 
\begin{equation}\label{eq:cjt}
r_3(G)=\min\left\{
|F|+
\val(\MX)\right\}
\end{equation}
where the minimum
is taken over all $F\subseteq E$ and all $2$-thin, $4$-shellable covers $\MX$ of $E\sm F$ with sets of cardinality  at least five.
\end{conj}
\noindent
Note that every $4$-shellable cover is $\MR_3$-degenerate since every graph on at most four vertices is $\MR_3$-independent

Clinch, Jackson and Tanigawa \cite[Theorems 6.3 and 6.1]{CJT} show that Conjectures \ref{con:dress} and \ref{con:cover} both hold if we replace the 3-dimensional rigidity matroid by another matroid on $\binom{V}{2}$, the {$\MC_2^1$-cofactor matroid}, which features in the study of bivariate splines.  
More generally, for any $d\geq 1$ and any graph $G=(V,E)$, the {\em $\MC_{d-1}^{d-2}$-cofactor matroid of $G$}, $\MC_{d-1}^{d-2}(G)$, is the row matroid of the $|E|\times d|V|$ matrix 
obtained by choosing a reference ordering for $V$ and a generic map $p:V\to \R^2$, and taking  the row of the matrix indexed by  $e=uv$ with $u<v$  to be:
\[
\kbordermatrix{
  & &  u & & v & \\
 e=uv & 0 \dots 0 & g_d(p_u-p_v) & 0\dots 0 & -g_d(p_u-p_v) & 0\dots 0
}
\]
where $g_d:(x,y)\mapsto (x^{d-1},x^{d-2}y,\ldots,y^{d-1})$.
Let $c_{d-1}^{d-2}(G)$ denote the rank of $\MC_{d-1}^{d-2}(G)$.  A {\em $C_{d-1}^{d-2}$-rigid-cluster} of $G$ is a maximum clique in the `$\MC_{d-1}^{d-2}$-closure' of $G$.

\begin{theorem}\label{thm:cjt}
For any graph $G$ 
\begin{equation}\label{eq:cof}
c^1_2(G)=\min\left\{
|F|+
\val(\MX)\right\}
\end{equation}
where the minimum
is taken over all $F\subseteq E$ and all $2$-thin, $4$-shellable covers $\MX$ of $E\sm F$ with sets of cardinality  at least five. In addition, equality holds in \eqref{eq:cof} when we take $\MX$ to be the family of 
$C_2^1$-rigid-clusters of $G$ of size at least five and $F$ to be the set of edges of $G$ which are not covered by $\MX$.
\end{theorem}

The $\MC_2^1$-cofactor matroid has many properties in common with the 3-dimensional rigidity matroid and the two matroids are conjectured to be identical by Whiteley in \cite{Wchapter}. Whiteley's conjecture, and hence also Conjectures \ref{con:dress} and \ref{con:cjt}, would follow from the following conjecture by a simple inductive argument.

\begin{conj}\label{con:cjt} Suppose $G=(V,E)$ is a graph and $v\in V$  with $N_G(v)=\{v_1,v_2,v_3,v_4,v_5\}$.\\
(a) If $v_1v_2,v_3v_4\notin E$ and $G-v+v_1v_2+v_3v_4$ is $\MR_3$-independent,
then $G$ is $\MR_3$-independent.\\
(b) If $v_1v_2,v_2v_3,v_iv_j,v_jv_k\not\in E$  for some $v_i,v_j,v_k\in N_G(v)$ with $v_j\neq v_2$, and $G-v+v_1v_2+v_2v_3$  and $G-v+v_iv_j+v_jv_k$ are both $\MR_3$-independent, then $G$ is $\MR_3$-independent.
\end{conj}

We refer to the operation which constructs $G$ from $G-v+v_1v_2+v_3v_4$ in (a) as the  {\em X-replacement operation}, and the operation which constructs $G$ from $G-v+v_1v_2+v_2v_3$  and $G-v+v_iv_j+v_jv_k$ in (b) as the  {\em double-V-replacement operation}. We can use a simple special position argument to show that the X-replacement operation preserves $C^1_2$-independence, see \cite{Wchapter}. This fact is used in \cite{CJT0} to deduce that  the double-V-replacement operation also preserves $C^1_2$-independence.

Theorem \ref{thm:cjt} follows from a different expression  for the rank function of the $C_2^1$-cofactor matroid given as \cite[Theorem 5.7]{CJT}. 
Given a graph $G=(V,E)$,  a {\em $K_5$-cover of $G$} is a family of edge sets of copies of $K_5$ in $K_V$, which cover $E$. A $K_5$-cover is {\em proper}  if it has an ordering $(C_1,C_2,\ldots,C_t)$ such that $C_i\not\subseteq \bigcup_{j=1}^{i-1}C_j$ for all $2\leq i\leq t$.

\begin{theorem}\label{thm:cjtproper}
For any graph $G$ 
\begin{equation}\label{eq:cof1}
c^1_2(G)=\min\left\{
|F|+\mbox{$\left|\bigcup_{C\in \MY} C\right|$}-|\MY|
\right\}
\end{equation}
where the minimum is taken over all $F\subseteq E$ and all proper 
$K_5$-covers $\MY$ of $G-F$.
\end{theorem}
\noindent
This result can also be viewed as a characterisation of $c_2^1(G)$ in terms of `$\MC^1_2$-degenerate covers' of $E$ with vertex sets of cardinality 5 by putting   $\MX=\{V(C):C\in \MY\}$ and noting that
$\val \,\MX=|\bigcup_{C\in \MY} C|-|\MY|$.

Theorems \ref{thm:cjt} and \ref{thm:cjtproper} each imply that the problem of determining whether a graph $G=(V,E)$ satisfies $c^1_2(G)=3|V|-6$, belongs to NP $\cap$ coNP. On the other hand, there is no known polynomial algorithm for determining the value of the functions on the right hand side of \eqref{eq:cof} or\eqref{eq:cof1}.

\subsection{Rigidity in dimensions at least four}\label{sec:d4}
The problem of characterising when a graph is rigid in $\R^d$ seems to be significantly more difficult when $d\geq 4$. This is signified by the results of Bolker and Roth \cite{BR} which imply that  $K_{d+2,d+2}$ is a closed, non-rigid circuit in the $d$-dimensional rigidity matroid whenever  $d\geq 4$. This in turn implies that the $d$-dimensional X-replacement operation does not preserve $\MR_d$-independence when $d\geq 4$ (since $K_{d+2,d+2}$ can be constructed from $K_2$  by a sequence of  $d$-dimensional 0- and 1-extension operations and an X-replacement).  In addition, the $d$-dimensional version of  Conjecture \ref{con:dress} fails, since every $\MR_d$-rigid cluster of  $K_{d+2,d+2}$  is a copy of $K_2$. 

On the other hand, it is at least conceivable that a modified version of Theorem \ref{thm:cjtproper} may hold for $\MR_d$.  A {\em $\{K_{d+2}, K_{d+2,d+2}\}$-cover of  a graph $G=(V,E)$} is a family of edge sets of subgraphs of $K_V$ each of which is a copy of $K_{d+2}$ or  $K_{d+2,d+2}$. The following conjecture appeared  in the special case when $d=4$ as \cite[Conjecture 8.2]{CJT} and was then extended to arbitrary $d$ in \cite[Conjecture 6]{JTmax}.

\begin{conj}\label{con:cjtproper}
For any graph $G$ 
\begin{equation}\label{eq:cof2}
r_d(G)=\min\left\{
|F|+\mbox{$\left|\bigcup_{C\in \MY} C\right|$}-|\MY|
\right\}
\end{equation}
where the minimum is taken over all $F\subseteq E$ and all proper 
$\{K_{d+2}, K_{d+2,d+2}\}$-covers $\MY$ of $G-F$.
\end{conj}

\subsection{Global Rigidity}
Connelly \cite{Con91} showed that the family of complete bipartite graphs provide examples which show that the  necessary conditions for generic global rigidity in $\R^d$ given in Theorem \ref{thm:hend} are not sufficient when $d\geq 3$.  The following  characterization of globally rigid complete bipartite graphs is stated without proof in \cite[Theorem 63.2.2]{JW}. A proof can be found at \cite{Jegres}.

\begin{theorem}
The complete bipartite graph $K_{s,t}$ is globally rigid in $\R^d$ if and only if $s,t\geq d+1$ and $s+t\geq \binom{d+2}{2}+1$.
\end{theorem}

This implies that $K_{5,5}$ is not globally rigid in $\R^3$ and, more generally, gives a finite family of graphs which satisfy the necessary conditions for global rigidity given in Theorem \ref{thm:hend}  
and are not globally rigid in $\R^d$, for any fixed $d$. 
Infinite families of such graphs were constructed by Frank and Jiang~\cite{FJ}  for each $d\geq 5$ and subsequently by Jord\'an, Kir\'aly and Tanigawa~\cite{JKT} for each $d\geq 3$. The smallest graph in their family for $d=3$ is shown in Figure \ref{fig:C6hinge}. To see that this graph $G$  is not globally rigid in $\R^3$, consider the graph $G'$ obtained from $G$ by replacing one of the $K_5$-subgraphs $H$ by a $K_4$-subgraph $H'$  with the same vertices of attachment as $H$. Let $F$ be the set of four edges of $H'$ which do not belong to any of the $K_5$-subgraphs of $G'$, and let $e\in F$. Then we can use \eqref{eq:degen} and the 2-thin,  $4$-shellable cover $\MX$ of $G'-F$ consisting of the five $K_5$-subgraphs  to show that $G'-e$ is not rigid in $\R^3$: we have $r_3(G'-e)\leq |F-e|+\val \,(\MX)=3+5\times 9-4=44=3|V(G')|-7$. This implies  that $G'$ is not redundantly rigid in $\R^3$ and hence is not globally rigid in $\R^3$ by Theorem \ref{thm:hend}. We can now use Lemma \ref{lem:sub}(b)
to deduce that  $G$ is not  globally rigid in $\R^3$. 

Lemma \ref{lem:sub}(b) also tells us that we can construct an infinite family of 4-connected graphs which are redundantly rigid but not globally rigid in $\R^3$ by replacing  a $K_5$-subgraph of the graph $G$ in Figure \ref{fig:C6hinge} by any  graph on at least five vertices which is globally rigid in $\R^3$.

\begin{figure}[t]
\begin{center}
\includegraphics[scale=0.6]{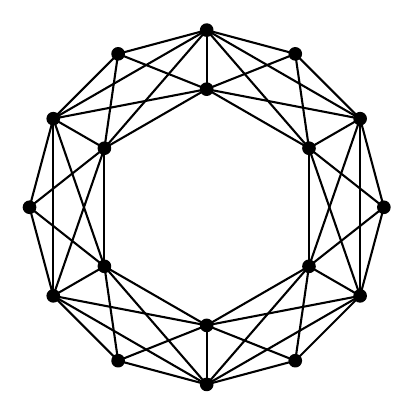}
\vspace{-0.5cm}
\end{center}
\caption{A $4$-connected graph which is redundantly rigid  but
not globally rigid in $\R^3$.}
\label{fig:C6hinge}
\end{figure}

The above discussion suggests the following new necessary condition for global rigidity in $\R^3$.

\begin{lemma}\label{lem:glob} Let $G=(V,E)$ be a graph on at least five vertices which is globally rigid in $\R^3$.  Suppose there exists a set   $F\subseteq E$  and a $3$-thin, $4$-shellable cover $\MX$ of $E\sm F$ with sets of cardinality at least five  such that $|F|+\val\, (\MX)=3|V|-6$. Then $|F|=\emptyset$ and $|\MX|=1$.
\end{lemma}

\begin{proof} Theorem \ref{thm:hend} and the hypothesis that $G$ is a globally rigid graph with at least five vertices imply that $G$ is 4-connected and redundantly rigid in $\R^3$. 

We first show that $F=\emptyset$. Suppose, for a contradiction, that  $F\neq \emptyset$ and choose $e\in F$. Then $|F-e|+\val\, (\MX)=3|V|-7$ and hence $G-e$ is not rigid in $\R^3$ by \eqref{eq:degen}, contradicting the redundant rigidity of $G$. Hence $F=\emptyset$.

To prove the second part of the lemma we assume, for a contradiction, that  $|\MX|\geq 2$. Let $(X_1,X_2,\ldots,X_t)$ be a `4-shellable ordering' of  $\MX$ and  put
$U=X_t\cap \bigcup_{i=1}^{t-1}X_i$.
 Then $|U|=4$ since $G$ is 4-connected and $X_t$ is the last set in the 4-shellable ordering of $\MX$. 

Let $F'$ be the  set of all edges of $K_U$ which are not hinges of $\MX$ and let $G'$ be obtained from $G$ by replacing the subgraph induced by $X_t$ by  $K_U$. Then $\MX':=\MX-X_t$ is a 3-thin,  4-shellable cover of $G'-F'$ and $|F'|+\val \, (\MX')=3|V(G')|-6$. If $F'\neq \emptyset$ then we could use the first part of the lemma to deduce that $G'$ is not globally rigid in $\R^3$. This would contradict Lemma \ref{lem:sub}(b), since $G$ is globally rigid and $G'$ is obtained from $G$ by replacing $G[X_t]$ by a copy of $K_4$. Hence $F'=\emptyset$ and every edge of $K_U$ is a hinge of $\MX$. Since $\MX$ is 3-thin this implies that $|\MX|\geq 3$, and hence $|\MX'|\geq 2$. We can now deduce by induction on $|V|$ that $G'$ is not globally rigid in $\R^3$. This again contradicts Lemma \ref{lem:sub}(b) and completes the proof of the lemma.
\end{proof}

Note that the lemma becomes false if we remove the hypothesis that $\MX$ is 3-thin. Consider the graph $G$ obtained by taking the union of two copies of  $K_5$ with four vertices in common. Then $G$ is globally rigid in $\R^3$ and the cover $\MX$ of $G$  consisting of the vertex sets of each of the $K_5$-subgraphs is $\MR_3$-degenerate and satisfies $\val\,(\MX)=3|V|-6$.

It is conceivable that the necessary condition for global rigidity given by Lemma \ref{lem:glob} is also sufficient for all graphs other than $K_{5,5}$. 

\begin{conj}
Let $G=(V,E)$ be a graph on at least five vertices which is not a copy of $K_{5,5}$. Then $G$ is globally rigid in $\R^3$ if and only if, for all $F\subseteq E$ and all 
3-thin, 4-shellable covers of $E\sm F$, 
$$|F|+\val\, (\MX)\geq 3|V|-6$$ 
with equality only when  $F=\emptyset$ and $\MX=\{V\}$.
\end{conj}

If true, this conjecture would imply that the complete bipartite graph $K_{5,5}$ is the only 5-connected graph which is redundantly rigid but not globally rigid in $\R^3$, and hence give an affirmative answer to \cite[Problem 4.2]{GGJ}.

To our knowledge, there are no conjectured combinatorial characterisations of graphs which are globally rigid in $\R^d$ when $d\geq 4$. This makes the following recent result of Vill\'anyi \cite{V} even more impressive.

\begin{theorem}
\label{hcr}
Every $d(d+1)$-connected graph is globally rigid in $\R^d$.
\end{theorem}

This result answers a conjecture of  Connelly, Jord\'an and Whiteley \cite{CJW} and also resolves a longstanding conjecture of Lov\'asz and Yemini \cite{LY} for rigidity in $\R^d$. The connectivity hypothesis in the theorem is best possible since there exist $(d(d+1)-1)$-connected graphs which are not even rigid in $\R^d$.

Another interesting recent development in generic (global) rigidity is the result by Lew et al.~\cite{LNPR}, who show that the threshold probabilities for rigidity and global rigidity of Erdős–Rényi random graphs coincide with that for having minimum degree at least $d$ and $d+1$, respectively. Their proof cleverly exploits properties of the rigidity matroid of the Erd{\"o}s-R{\'e}nyi random graph.

\section{Applications to other areas of combinatorics}

\subsection{Orientations and packings of graphs}

Tools from rigidity theory, along with matroid and graph theoretic methods,
have recently been used  to solve two long standing open problems in graph theory
concerning highly connected orientations and removable spanning trees, respectively.
It follows from a classical theorem of Nash-Williams 
\cite{NWori}
that every $2k$-edge-connected graph has a $k$-arc-connected orientation. In 1985, Thomassen \cite{thomassen_1989}
asked whether a similar statement is true for vertex-connectivity.
He conjectured that, 
for every positive integer $k$, there exists a (smallest) integer $f(k)$ such that every $f(k)$-connected graph has a $k$-connected orientation.
It is easy to see that if $f(k)$ exists, then $f(k) \geq 2k$. 
It follows from the above mentioned result of Nash-Williams that $f(1) = 2$.  Jord\'an\cite{Jdisjoint} showed that $f(2) \leq 18$. His proof uses the
following result on packing rigid subgraphs, together with results on
$2$-connected orientations of Eulerian graphs. Thoughout this section, we will say that a graph $G$ is {\em $d$-rigid} if it is (generically) rigid in $\R^d$,
and that $G$ is {\em minimally $d$-rigid} if it is $d$-rigid but $G-e$ is not $d$-rigid for all edges $e$ of $G$.

\begin{theorem}\label{theorem:jordan}
    Every $6t$-connected graph contains $t$ edge-disjoint $2$-rigid (and hence $2$-connected) spanning subgraphs. 
\end{theorem}

The proof of Theorem \ref{theorem:jordan} used Theorem \ref{thm:PG} and 
(\ref{eq:LY}) 
together with
Edmonds' formula for the rank function of unions of matroids \cite{Esum}. It also showed that
a packing of $2$-rigid spanning subgraphs can be found in polynomial time.

Subsequently, Thomassen \cite{thomassen_2015} proved that $f(2) = 4$. Durand de Gevigney \cite{duranddegevigney_2020} showed that for $k\geq 3$, deciding whether a graph has a $k$-connected orientation is NP-hard.

The proof of Theorem \ref{hcr} showed that sufficient connectivity conditions for
$d$-rigidity can be obtained without knowing the
rank function of ${\cal R}_d(G)$. Soon after the appearance of Vill\'anyi's probabilistic proof, a similar approach was used 
by Garamv\"olgyi, Jord\'an, Kir\'aly and Vill\'anyi \cite{GJKV} to obtain the following packing result, which
can be considered as a result on the matroid union of $t$ copies of the $d$-dimensional
generic rigidity matroid. 

\begin{theorem}\label{theorem:main1}
    Every $\left(t \cdot 10d(d+1)\right)$-connected graph contains $t$ edge-disjoint $d$-rigid (and hence $d$-connected) spanning subgraphs.
\end{theorem}

The existence of a constant $h(d)$ such that every $(t \cdot h(d))$-connected graph contains $t$ edge-disjoint $d$-connected spanning subgraphs was conjectured in \cite{garamvolgyi.etal_2024}.
The bound on $h(d)$ is Theorem \ref{theorem:main1} is certainly not tight.
It is conceivable, that $h(d) \leq d(d+1)$. For $t = 1$, this is
Theorem \ref{hcr}.

Theorem \ref{theorem:main1} was combined with our next two lemmas in \cite{GJKV} to solve Thomassen's orientation problem.  
 Given a graph $G = (V,E)$ and a function  $g:V \to \mathbb{Z}_+$, we shall use the notation $g(X) = \sum_{v \in X} g(v)$ for subsets $X \subseteq V$.
 The in-degree of a subset $X$ of vertices in a digraph $D$ is denoted by
 $\rho_D(X)$.
Hakimi \cite{hakimi_1965} proved that 
$G$ has an orientation $\DD$ in  
$\rho_\DD(v) = g(v)$ for all $v\in V$ if and only if
(i) $i_G(X)\leq g(X)$ for all nonempty $X\subseteq V$, and
  (ii)  $|E|=g(V)$ hold. 

We may use Hakimi's result and the sparsity count of Lemma \ref{lem:max} to show the existence of
certain in-degree specified orientations of minimally $d$-rigid graphs. 
Given a minimally $d$-rigid graph $G=(V,E)$ with $|V|\geq \binom{d+1}{2}$ and a subset
$R\subseteq V$ with $|R|=\binom{d+1}{2}$, we define the 
in-degree specification function $g_{d,R}$ by putting
$g_{d,R}(v)=d$ for all $v\in V-R$ and $g_{d,R}(r)=d-1$ for all $r\in R$.
We say that an orientation $\DD$ of $G$ is a
{\it $(d,R)$-orientation} if its in-degrees respect the specification $g_{d,R}$.

\begin{lemma}
\label{l1}
Let $G=(V,E)$ be a minimally $d$-rigid graph with $|V|\geq \binom{d+1}{2}$ and let 
$R\subseteq V$ be a set of vertices with $|R|=\binom{d+1}{2}$. Then $G$ has a
$(d,R)$-orientation.
\end{lemma}

The second lemma provides a lower bound on the number of in-neighbors of certain subsets in $(d,R)$-orientations, 
establishing a link between degree-specified orientations and high vertex-connectivity.

\begin{lemma}
\label{l2}
Let $d$ and $k$ be integers with $k \geq 2$ and $d \geq 4k-4$. Let $G=(V,E)$ be a minimally $d$-rigid graph with $|V|\geq \binom{d+1}{2}$, 
$R \subseteq V$ a set of vertices with $|R|=\binom{d+1}{2}$, and let 
$\DD$ be a $(d,R)$-orientation of $G$. Finally, let $X \subseteq V$ be a set of vertices. If $|X\cap R|\leq \frac{{\binom{d+1}{2}}}2$,
then $X$ has at least $k$ in-neighbors in $\DD$.
\end{lemma}

These lemmas were combined with  Theorem \ref{theorem:main1}  in \cite{GJKV} to give:

\begin{theorem}\label{theorem:main2}
    Every $(320 \cdot k^2)$-connected graph has a $k$-connected orientation.
\end{theorem}
\begin{sproof}
    Suppose that $k\geq 2$ and let 
$d=4k-4$.
By Theorem \ref{theorem:main1} 
$G = (V,E)$ has
two edge-disjoint minimally $d$-rigid spanning subgraphs $G_1$ and $G_2$.

Let us fix a set $R \subseteq V$ of vertices with $|R|=\binom{d+1}{2}$.
The desired orientation $\DD$ of $G$ is defined by defining the orientations of $G_1$ and $G_2$, and then orienting the remaining edges arbitrarily.
The orientation of $G_1$ is chosen to be a $(d,R)$-orientation.
The orientation of $G_2$ is chosen to be a {\em reversed} $(d,R)$-orientation.
By Lemma \ref{l1}, these orientations exist.
It is not hard to show, using Lemma \ref{l2}, that 
the union $\DD$ of these oriented spanning subgraphs
is $k$-connected.
\end{sproof}

The bound on $f(k)$ given by Theorem \ref{theorem:main2} is probably far from being tight. In particular, it is still open whether $f(k) = 2k$ holds.

One may also consider the rooted version of the above orientation problem.
A digraph $D$ is  {\it rooted $k$-connected} (from a given root vertex $r$), if it contains $k$ internally disjoint directed paths from $r$ to each vertex 
$v\in V(D)-r$. Rooted $k$-connected graphs are defined analogously. 
It was recently conjectured
in \cite{Jproc} that sufficiently highly
rooted connected graphs have rooted $k$-connected orientations.
It was shown in \cite{Jproc} that every rooted $6$-connected graph has a rooted $2$-connected orientation, and such an orientation can be found in polynomial time.
The proof of this result also employs methods from 
rigidity theory. In particular, a key step of the proof is to show that 
if $G$ is rooted $6$-connected from $r$ and $G-r$ is $2$-edge connected, then
$G$ is redundantly rigid in $\RR^2$, which strengthens a result of \cite{LY}. 

The packing Theorem \ref{theorem:main1} immediately implies the solution of another
conjecture on graph connectivity concerning removable spanning trees.
Kriesell \cite[Problem 444]{mohar.etal_2007} conjectured in 2003 that 
for every positive integer $k$ there exists a (smallest) integer $g(k)$ such that every $g(k)$-connected graph $G$ contains a spanning tree $T$ for which $G - E(T)$ is $k$-connected.

As with Thomassen's conjecture, the edge-connected version of Kriesell's
conjecture is a well-known result: it follows from a classical theorem of Nash-Williams 
\cite{nash-williams_1961} 
and Tutte 
\cite{tutte_1961} 
that every $(2k+2)$-edge-connected graph $G$ contains a spanning tree $T$ such that $G - E(T)$ is $k$-edge-connected.
In particular, we have $g(1) = 4$.
Theorem \ref{theorem:jordan} implies that $g(2)\leq 12$.
The $k=3$ case was settled by   Garamv\"olgyi, Jord\'an, and Kir\'aly \cite{garamvolgyi.etal_2024} using the matroid union of  $t$ copies of the  $\mathcal{C}^1_2$-cofactor matroid and Theorem \ref{thm:cjt}.

\begin{theorem}\label{theorem:cofactor}
Every $12t$-connected graph contains $t$ edge-disjoint $\mathcal{C}^1_2$-rigid (and hence $3$-connected) spanning subgraphs. In particular, $g(3) \leq 24$.
\end{theorem}

By applying
Theorem \ref{theorem:main1} with $t=2$ and choosing a spanning tree
of one of the two $d$-rigid spanning subgraphs, we get
$g(d) \leq 20d(d+1)$. To improve upon this result, one can consider the union of the $d$-dimensional rigidity matroid and the graphic matroid. 
It was shown in \cite{GJKV} that
this leads to the bound $g(d) \leq d^2 + 3d + 5$.

\begin{theorem}\label{theorem:Kriesellimproved}
    Every $(d^2 + 3d + 5)$-connected graph has a 
    spanning tree $T$ for which $G - E(T)$ is $k$-connected.
\end{theorem}

The bound given by Theorem \ref{theorem:Kriesellimproved} is still not optimal. Perhaps every $(d(d+1)+2)$-connected graph contains edge-disjoint copies of a $d$-rigid spanning subgraph and a spanning tree. Again, this bound would be tight. 

An immediate corollary of Theorem \ref{theorem:Kriesellimproved} is that if $G$ is sufficiently highly connected then,
for all $s,t\in V(G)$, there exists a path $P$ from $s$ to $t$ in $G$ such that $G-E(P)$ is $k$-connected.
The existence of such paths was verified earlier in \cite{kawarabayashi.etal_2008}, assuming that $G$ is $(1600k^4+k+2)$-connected. 
 This connectivity requirement can be substantially weakened using Theorem \ref{theorem:Kriesellimproved}.

\subsection{Rigidity and simplicial complexes}

An {\em abstract simplicial complex} $\Delta$ is a finite family of sets with the property that $\tau \subseteq \sigma \in \Delta $ implies that $\tau \in \Delta$. An element $\sigma \in \Delta$ is called a {\em face of dimension $|\sigma| -1$},  or a {\em $(|\sigma|-1)$-face} for short, and we define $\dim(\Delta) = \max\{\dim \sigma: \sigma \in \Delta\}$. A {\em facet} of $\Delta$ is a maximal face and we say that $\Delta$ is {\em pure} if all its facets have the same dimension.  We write $\MF_j(\Delta)$ for  the set of $j$-faces of $\Delta$ and put $f_j(\Delta) = |\MF_j(\Delta)|$. The {\em $f$-vector of $\Delta$} is $f(\Delta) = (f_0,f_1,\ldots)$.
The {\em graph of $\Delta$} is $G(\Delta) = (\MF_0(\Delta), \MF_1(\Delta))$ and so elements of $\MF_0(\Delta)$, respectively $\MF_1(\Delta)$ are called {\em vertices}, respectively {\em edges}, of $\Delta$. Abusing notation, we will write $v$ instead of $\{v\}$ for the vertex whose single element is $v$ and, similarly $uv$ instead of $\{u,v\}$. The {\em link} of a face $\sigma \in \Delta$ is defined by $\lk_\Delta (\sigma) = \{\tau\setminus\sigma: \tau \in \Delta, \tau \supseteq \sigma\}$.

\subsubsection{Simplicial circuits and Fogelsanger's Theorem}

Given a finite set $X$ and positive integer $k$, the  {\em $k$-simplicial matroid} $\MS_k(X)$ is the row matroid of the $\binom{|X|}{k+1}\times \binom{|X|}{k}$ matrix $I$, defined over $\ZZ_2$, in which the rows are indexed by  $\binom{X}{k+1}$, the columns are indexed by $\binom{X}{k}$ and the $(\sigma,\tau)$-entry  is $1$ if $\tau \subset \sigma$ and $0$ otherwise. 
 Clearly $\MS_1([n])$ is the graphic matroid on $K_n$. See \cite[Chapter 10]{welsh} for more information on simplicial matroids.\footnote{In Welsh's notation the $(k+1)$-simplicial matroid is what we call $\MS_k(X)$. Our notation is more in line with the topological convention that the dimension of a face $\sigma$ is $|\sigma|-1$.}

A {\em simplicial $k$-circuit} is a pure $k$-dimensional simplicial complex $\Delta$ such that $\MF_k(\Delta)$ is a circuit in $\MS_k(\MF_0(\Delta))$. For example, a cycle graph is a simplicial 1-circuit. More explicitly we have the following characterisation. 

\begin{lemma}
    \label{lem:circuit}
    An abstract simplicial complex $\Delta$ is a simplicial $k$-circuit if and only if $\MF_k(\Delta) \neq \emptyset$, every $(k-1)$-face of $\Delta$ is contained in an even number of $k$-faces, and no proper subcomplex of $\Delta$ has these properties. 
\end{lemma}

In particular, for $d\geq2$, the boundary complex of simplicial convex polytope  is a simplicial $(d-1)$-circuit.
A more general example is the following. A {\em $k$-pseudomanifold} is a pure $k$-dimensional simplicial complex $\Pi$ such that 
\begin{enumerate}
    \item every $(k-1)$-face of $\Pi$ belongs to exactly two $k$-faces of $\Pi$, and 
    \item $\Pi$ is strongly connected i.e. if $\sigma,\sigma' \in \MF_k(\Pi)$ then there is a sequence of $k$-faces $\sigma = \tau_1,\cdots,\tau_s = \sigma'$ such that $|\tau_i \cap \tau_{i+1}| = k$ for $i = 1,\cdots, s-1$.
\end{enumerate}
In particular, every triangulation of a $k$-dimensional manifold is a $k$-pseudomanifold.\footnote{Recall that a triangulation of a manifold is a simplicial complex which has a homeomorphism from its geometric realisation to the given manifold.} 
It follows easily from Lemma \ref{lem:circuit} that  a $k$-pseudomanifold is a simplicial $k$-circuit for all $k\geq 1$, although the converse is not necessarily true when $k\geq 2$.

\begin{theorem}[Fogelsanger's Rigidity Theorem]
    \label{thm:fog_rig}
    Let $\Delta$ be a simplicial $k$-circuit for some $k\geq 2$. Then $G(\Delta)$ is rigid in $\mathbb R^{k+1}$. 
\end{theorem}

The proof is an inductive argument based on edge contraction/vertex splitting. Given two vertices $u,v$  of a simplicial complex $\Delta$, we define $\Delta/uv$ to be the complex obtained from $\Delta$ by deleting every face that contains $\{u,v\}$ and then replacing every face $\sigma$ that contains $v$ by $\sigma \setminus\{v\} \cup \{u\}$. An immediate difficulty for the induction is that edge contraction does not necessarily preserve the property of being a simplicial circuit. Fogelsanger discovered a key decomposition argument to circumvent this problem. 

\begin{lemma}
    \label{lem:fog_decomp}
    Let $\Delta$ be a simplicial $k$-circuit for some $k \geq 2$, $uv \in \MF_1(\Delta)$ and suppose that $\Delta/uv$ is not a simplicial $k$-circuit. Then there exist simplicial $k$-circuits $\Delta_1, \cdots, \Delta_m$ such that 
    \begin{enumerate}
        \item[(F1)] $G(\Delta) = \bigcup_{i=1}^mG(\Delta_i)$,
        \item[(F2)] $uv \in \Delta_i$ for $1 \leq i \leq m$, and, either $\Delta_i$ is the boundary complex of a $(k+1)$-simplex or $\Delta_i/uv$ is a simplicial $k$-circuit,
        \item[(F3)] for $2 \leq j \leq m$, $\MF_k(\Delta_j) \cap \left(\bigcup_{i=1}^{j-1} \MF_k(\Delta_i)\right) \neq \emptyset$.
    \end{enumerate} 
\end{lemma}

In fact Fogelsanger proved more general versions of Theorems \ref{thm:fog_rig} and Lemma \ref{lem:fog_decomp} for {\em minimal homology cycles} over an arbitrary field. These coincide with our simplicial circuits when the field is $\ZZ_2$. See \cite{F} for details.

Next we outline the proof of Theorem \ref{thm:fog_rig},

\begin{sproof}
    Pick $uv \in \MF_1(\Delta)$.
    If $\Delta/uv$ is again a simplicial $k$-circuit then by induction on the number of vertices $G(\Delta)/uv = G(\Delta/uv)$ is rigid in $\R^{k+1}$. Now, since $\Delta$ is a simplicial $k$-circuit, $\{u,v\}$ is contained in at least two $k$-faces. This implies that $u$ and $v$ have at least $k$ common neighbours in $G(\Delta)$ and so, by Lemma \ref{lem:split}, $G$ is rigid in $\R^{k+1}$.

    Thus we can assume that $\Delta/uv$ is not a simplicial $k$-circuit. Let $\Delta_1,\cdots,\Delta_m$ be the simplicial $k$-circuits whose existence is asserted by Lemma \ref{lem:fog_decomp}. By (F2) and using the induction argument of the previous paragraph, we see that $G(\Delta_i)$ is rigid in $\R^{k+1}$ for each $i$. Now (F3) implies that, for $2\leq j \leq m$, $\bigcup_{i=1}^{j-1}G(\Delta_i)$ and $G(\Delta_j)$ have at least $k+1$ common vertices so an induction argument using Lemma \ref{lem:glue} and (F1) proves that $G(\Delta)$ is rigid in $\R^{k+1}$.
\end{sproof}

The following result on the  global rigidity of simplicial $k$-circuits was recently obtained in \cite{CJt-global-mfd}. 

\begin{theorem}
    \label{thm:CJTglobal}
    Let $\Delta$ be a simplicial $k$-circuit for $k\geq 2$. 
    \begin{enumerate}
        \item If $k\geq 3$ then $G(\Delta)$ is globally rigid in $\R^{k+1}$ if and only if $G(\Delta)$ is $(k+2)$-connected or is a copy of $K_{k+2}$.
        \item If $k=2$ then $G(\Delta)$ is globally rigid in $\R^3$ if and only if $G(\Delta)$ is $4$-connected and not planar, or, $G(\Delta)$ is a copy of $K_4$.
    \end{enumerate} 
\end{theorem}

The proof of Theorem \ref{thm:CJTglobal} is again an induction based on vertex splitting and on Fogelsanger's decomposition. However there are significant extra complications that arise from controlling the $(k+2)$-connectedness in the inductive arguments, and also from the fact that, as mentioned in Section \ref{sec:induct}, the natural analogue of Lemma \ref{lem:split} for global rigidity is still a conjecture. For more details see \cite{CJt-global-mfd}.

Theorems \ref{thm:fog_rig} and \ref{thm:CJTglobal} have some important consequences for the theory of $f$-vectors of simplicial $k$-polytopes and more generally of $k$-pseudomanifolds. It is natural to wonder what is the smallest possible number of $j$-faces in a simplicial $k$-polytope with a given number of vertices.
For $0 \leq j \leq k \leq n-2$ define 
$$
\phi_j(k,n) = \left\{\begin{array}{ll}\binom{k+1}j n  - j\binom{k+2}{j+1} & \text{ if }j < k\\ 
kn - (k-1)(k+2)&\text{ if }j = k. \end{array}\right. 
$$
The number $\phi_j(k,n)$ is the numbers of $j$-faces in a {\em stacked $k$-spheres}. This family of $k$-polytopes can be defined recursively as follows. The smallest 
stacked $k$-sphere 
is the boundary complex of the $(k+1)$-dimensional simplex. 
For $n\geq k+3$ a stacked $k$-sphere with $n$ vertices is a simplicial complex obtained by deleting one of the facets of a stacked $k$-sphere with $n-1$ vertices and adding a new vertex whose link is the boundary complex of the deleted facet. An easy induction argument proves that if $\Sigma$ is a stacked sphere with $n$ vertices then $f_j(\Sigma) = \phi_j(k,n)$. The following theorem is due to Barnette \cite{B71,B73}. 
\begin{theorem}[The Lower Bound Theorem]
    \label{thm:lbt}
    Let $k \geq 2$ be an integer and  $\Delta$ be the boundary complex of a simplicial $k+1$ dimensional polytope with $n$ vertices. Then, for $0\leq j \leq k$,
    \begin{equation}
    \label{eqn:lbt}
    f_j(\Delta) \geq \phi_j(k,n).
    \end{equation}
    Moreover, if $k \geq 3$ and equality holds in (\ref{eqn:lbt}) for some $1 \leq j \leq k$ then $\Delta$ is a stacked $k$-sphere. If $k=2$ then equality holds if and only if $\Delta$ is a triangulation of the 2-sphere.
\end{theorem}

\begin{sproof}
    We outline a proof of the inequality using rigidity that is due to Kalai \cite{K}. The use of global rigidity to characterise the case of equality is taken from \cite{CJt-global-mfd}.
    
    It is well known that the link of a vertex of $\Delta$ is 
    the boundary complex of a simplicial polytope of dimension $k$. Based on this fact, McMullen, Perles and Walkup \cite{McMWal} gave an inductive argument to show that we can reduce the proof of 
    the theorem to the case $j = 1$. In this case (\ref{eqn:lbt}) becomes
    \begin{equation}\label{eqn:lbtcase1} f_1(\Delta) \geq (k+1)n - \binom{k+2}2.\end{equation}
    Let $p:\MF_0(\Delta) \to \R^{k+1}$ be the embedding of the vertices that arises naturally from the given polytope. Note that sufficiently small perturbations of a simplicial polytope do not change its boundary complex, so we assume that $p$ is generic from now on.
    From Theorem \ref{thm:fog_rig} the framework $(G(\Delta),p)$ is 
    infinitesimally rigid. Now (\ref{eqn:lbtcase1}) follows immediately from the fact that 
    equality holds in (\ref{eq:max}) for infinitesimally rigid frameworks.

    To deal with with the case of equality when $k\geq 3$, we observe that if equality holds in (\ref{eqn:lbtcase1}) then the framework
    of the polytope is minimally rigid. If $f_0(\Delta) \geq k+2$ then Theorem \ref{thm:hend} implies that $(G(\Delta),p)$ is not globally rigid. Since $p$ is generic Theorem \ref{thm:CJTglobal} implies that 
    $G(\Delta)$ 
    is not $(k+2)$-connected. Now it can be be shown that in this case $\Delta$ is obtained by gluing two smaller complexes together along a common facet, whose set of vertices is a $(k+1)$-vertex separator for $G(\Delta)$, and removing the common facet. It also follows that each of these two smaller complexes is the boundary complex of a polytope and that equality in (\ref{eqn:lbtcase1}) holds in both cases, so by induction each of these complexes is a stacked sphere and therefore $\Delta$ is also a stacked sphere.
    The case of equality for $k=2$ follows by a similar argument.
\end{sproof}

\paragraph{Some historical comments.} Rigidity of frameworks associated to polytopes  is a very classical topic, dating at least to Cauchy's Rigidity Theorem for 3-dimensional convex polyhedra\cite{Cauchy}.
Dehn proved an infinitesimal rigidity version of Cauchy's Theorem \cite{Dehn}, and later  Alexandrov extended Cauchy's result to higher dimensions \cite{Alexandrov}. Gluck extended Dehn's result to generic rigidity in $\R^3$ of plane triangulations \cite{G}. However, in a surprising development, Connelly showed that there is a non-convex polyhedron, i.e. an embedded simplicial $2$-sphere in $\R^3$, that is continuously flexible \cite{C82}.
Whiteley extended Alexandrov's theorem by showing that, for all $d \geq 3$, if we triangulate the $2$-faces of a non-simplicial polytope without adding new vertices then the resulting framework is infinitesimally rigid \cite{WW-polyhedra}.
Kalai \cite{K} proved Theorem \ref{thm:fog_rig} for triangulations of 
manifolds of dimension at least $3$  and used it to extend the Lower Bound Theorem to triangulated manifolds. Tay subsequently extended the Lower Bound Theorem to the class of normal pseudomanifolds \cite{Tay-LBT}.
Following Kalai's  paper \cite{K},
rigidity properties and lower bound theorems have been investigated for several other classes of complexes including balanced complexes \cite{KleeNov,Oba}, cubical complexes \cite{KNN} and symmetric complexes \cite{KNNZ,CJT-symm}.
As far as we are aware the following is still open.
\begin{conj}
    Let $k \geq 2$ be an integer and 
    $\Delta$ be a simplicial $k$-circuit with $n$ vertices. Then $f_{j}(\Delta) \geq \phi_j(k,n)$ for all $j \geq 0$. Moreover, if equality holds for some $1 \leq j \leq k$ then $\Delta$ is a simplicial $k$-sphere.
\end{conj}

\subsubsection{Connections with commutative algebra}

In this section we will outline some connections between rigidity theory and combinatorial commutative algebra. We will assume that the reader is familiar with the basic concepts and terminology of commutative algebra. Since we are focused on the connection with rigidity, all of the rings that we discuss in this section will be $\R$-algebras. However many of the algebraic concepts and results are also valid over other fields. 
Let $\R[X] = \R[x_1,\cdots, x_n]$ be the polynomial ring in $n$ commuting variables. 
For  $\tau \subset [n]$  we will write $x_\tau$ for the {\em squarefree monomial} $\Pi_{i \in \tau}x_i$. The mapping $\tau \mapsto x_\tau$ is a bijection between nonempty subsets of $[n]$ and the set of squarefree monomials in $\R[X]$.

To a simplicial complex $\Delta$ with $\MF_0(\Delta) = [n]$ we associate an ideal of $\R[X]$ defined by
$$I_\Delta = (x_\tau: \tau \subset [n], \tau \not\in \Delta).$$ In words, $I_\Delta$ is the ideal of $\R[X]$ generated by squarefree monomials corresponding to non-faces of $\Delta$. The Stanley-Reisner ring of $\Delta$ is defined by 
$$\R[\Delta] = \frac{\R[X]}{I_\Delta}. $$
It is a commutative graded algebra whose $j$-homogeneous component  $\R[\Delta]_j$ is the vector space with a basis consisting of all monomials (including non-squarefree monomials) of total degree $j$ whose support is in $\Delta$. The Stanley-Reisner ring is one of the central objects in combinatorial commutative algebra, see \cite{Stanley,MillerSturmfels,HerzogHibi} for more comprehensive treatments.  

How does this connect with rigidity theory? Suppose that $p:[n] \to \R^{d}$. 
For $1\leq j \leq d$ let $\theta_j  = \sum_{i=1}^n p(i)_j x_i \in \R[\Delta]_1$ and write $\theta$ for the sequence $\theta_1,\ldots,\theta_{d}$. Let $\ell = x_1+\cdots +x_n$, Let $B = B(\Delta,\theta) =\frac{\R[\Delta]}{(\theta)}$ and let $C = C(\Delta,\theta) =\frac{\R[\Delta]}{(\theta,\ell)} $. Then $B$ and $C$ are graded $\R$-algebras.  Let $B_j$, respectively $C_j$ be the $j$-homogeneous component of $B$, respectively $C$. 
The connection with classical rigidity theory is given by the following result of Lee \cite{Lee}. Recall the definition of a stress matrix from Section \ref{sec:GRandStressMatrices}. 

\begin{theorem}
    \label{thm:lee}
    There is a natural isomorphism between $C_2$ and the vector space of all stress matrices of $R(G(\Delta),p)$.
\end{theorem}
\begin{sproof}
    The basis $\{x_ix_j+(I_\Delta)_2:ij \in \MF_2(\Delta)\}$  induces an inner product on $\R[\Delta]_2$.
    It can be shown that $$(\theta,\ell)_2^\perp  = \left\{\left(\sum_{1\leq i,j\leq n} \Omega_{ij}x_ix_j\right) +(I_\Delta)_2: \Omega\text{ is a stress matrix of } (G(\Delta),p)\right\}.$$
\end{sproof}
In fact there are similar interpretations of $C_j$ for $2 \leq j \leq \dim \Delta +1$ in terms of the skeletal rigidity matrices associated to the pair $(\Delta,p)$, see \cite{Lee} for further details. In the case $j = 1$, we have $B_1$, respectively $C_1$, is naturally isomorphic to  the space of linear, respectively affine, dependencies among the points $p(1),\ldots,p(n)$ and so, it has dimension $n-d$, respectively $n-d-1$, when $p$ is in general position.

Up to now our discussion has only required elementary concepts of commutative algebra. The remainder of this section will make reference to some more advanced notions from commutative algebra, see \cite{MillerSturmfels,BrunsHerzog} for more detail on these algebraic aspects. 
For our purposes it suffices to mention the following facts. 

Suppose that $p:\MF_0(\Delta) \to \R^d$ is generic with $\dim\Delta = d-1$.
It is not too difficult to show that $B_j = 0$ for $j\geq d+1$. In particular, this means that in this case $B$ is an {\em artinian 
reduction} of $\R[\Delta]$. 

The Cohen-Macaulay property is extensively studied in commutative algebra. 
Macaulay has shown $\R[\Delta]$ is a Cohen-Macaulay ring if and only if 
\begin{equation}
    \label{eqn:hi}
    \dim B_i = h_i(\Delta)  = \sum_{j=0}^i (-1)^{i-j} \binom{d-j}{i-j} f_{j-1}(\Delta).
\end{equation}
See \cite{Stanley78} for an account of this. Also we note that it follows from a theorem of Reisner that the Stanley-Reisner ring of a simplicial sphere is Cohen-Macaulay, see \cite[Corollary 5.3.9]{BrunsHerzog}. 

Using (\ref{eqn:hi}) and Theorem \ref{thm:lee} we obtain the following.
\begin{lemma}
    \label{lem:CMrig}
    Suppose that $\R[\Delta]$ is Cohen-Macaulay and $\dim \Delta = d-1 \geq 1$. Then $G(\Delta)$ is rigid in $\R^d$ if and only if multiplication by $\ell$ is an injective map $B_1 \to B_2$.
\end{lemma}

An element $m$ of a graded Artinian algebra $\mathcal{A}$ is said to have the {\em weak Lefschetz property} if $m \in \mathcal A_1$ and the multiplication map $x\mapsto mx$ from $\mathcal A_i$ to $\mathcal A_{i+1}$ has maximal rank for all $i$, i.e. this map is either injective or surjective. The weak Lefschetz property is of great interest in commutative algebra and algebraic geometry. 
Lemma \ref{lem:CMrig} indicates that  there is a connection between rigidity theory and the weak Lefschetz property for an artinian reduction of the Stanley-Reisner ring. This connection has been developed much further, see \cite{Tay96,TayWhiteWhiteleyI,TayWhiteWhiteleyII,TayWhiteley}. In particular we mention Adiprasito's recent proof of the $g$-theorem, which characterises the $f$-vectors of rational homology spheres. This breakthrough result proves, among other things, that artinian reductions of the corresponding Stanley-Reisner rings have the weak Lefschetz property. The proof makes crucial use of the coning property of skeletal rigidity \cite{adiprasito}.

Finally in this section we briefly discuss, without giving details, the theory of algebraic shifting. For a simplicial complex $\Delta$ with vertex set $[n]$, Kalai \cite{Kalai-shifting} constructs a complex $\Delta^s$ with the following properties.
\begin{enumerate}
    \item[(1)] $f_i(\Delta^s) = f_i(\Delta)$ for all $i \geq 0$.
    \item[(2)] If $j \in \sigma \in \Delta^s$ and $i<j$ such that $i \not\in \Delta^s$ then $\sigma \setminus \{j\} \cup \{i\} \in \Delta^s$.
    \item[(3)] For $d\geq 1$, $\rank \MR_d( G(\Delta^s)) =\rank \MR_d(G(\Delta))$.
\end{enumerate}
Roughly speaking the algebraic shifting operator $\Delta \to \Delta^s$ corresponds at the level of Stanley-Reisner ideals to the generic initial ideal construction - there are some technicalities to deal with the fact that generic initial ideals of squarefree monomial ideals need not be squarefree. 
Algebraic shifting has many other interesting algebraic and combinatorial properties. For example it is not difficult to deduce from (2) and (3) above that a graph $G$, considered as a 1-dimensional simplicial complex, with vertex set $[n]$ for some $n \geq d+1$ is rigid in $\R^d$ if and only if $\{d,n\} \in G^s$. We note that this does not yield a more efficient algorithm for determining rigidity since computing the generic initial ideal of a squarefree monomial ideal is computationally as difficult as computing the rank of the generic rigidity matrix. 

The basic construction of algebraic shifting was introduced in \cite{Kalai-hyperconnectivity} in the context of the hyperconnectivity matroid of a graph, which is a relative of the rigidity matroid. Indeed Kalai shows that analogous shifting operators can be constructed for any suitably symmetric matroid, and uses this to study growth functions of symmetric matroids, see \cite{Kalai-symm}. In his thesis Nevo developed the theory further and gave some applications to $f$-vectors of graded posets \cite{nevothesis}. We also note some recent work on volume rigidity and algebraic shifting \cite{BNP}.

\subsection{Abstract rigidity and matroid maximality problems}
We will describe Jack Graver's ingenious concept of abstract rigidity matroids, which now plays a central role in combinatorial rigidity and can be used to connect rigidity theory to other topics  in combinatorics such as weakly saturated sequences of graphs, tropical ideals and the problem of deciding if a given family of matroids contains a unique maximal element.


\subsubsection{Abstract rigidity}\label{sec:abstract}
A key observation of Graver~\cite{G91} for the introduction of abstract rigidity is that the gluing property,  one of the characteristic properties of graph rigidity  given in Lemma~\ref{lem:glue}(a),  can be described purely in terms of the closure operator of the rigidity matroid.

Let $M$ be a matroid on 
the edge set of the complete graph $K_n$ with $n\geq d+1$, 
and  closure operator 
${\rm cl}_M$. Then
$M$ is an {\em abstract $d$-rigidity matroid} if $M$ 
satisfies the following two axioms (which correspond to the `only-if' and `if' directions of Lemma~\ref{lem:glue}(a), respectively): 
\begin{description}
\item[(G1)] If $E_1, E_2\subseteq E(K_n)$ with $V(E_1)=V_1$, 
$V(E_2)=V_2$ and $|V_1\cap V_2|\leq d-1$, then
${\rm cl}_M(E_1\cup E_2)\subseteq\binom{V_1}{2}\cup \binom{V_2}{2}$;
\item[(G2)]  If $E_1, E_2\subseteq E(K_n)$ with $\cl_M(E_1)=\binom{V_1}{2}$, 
$\cl_M(E_2)=\binom{V_2}{2}$ and $|V_1\cap V_2|\geq d$, then\\
${\rm cl}_M(E_1\cup E_2)= \binom{V_1\cup V_2}{2}$.
\end{description}

It is known that Lemma~\ref{lem:glue} holds for any framework  $(G,p)$ such that                                                                                                                        $p$ is in general position in $\R^d$.
Hence the rigidity matroid of any $d$-dimensional framework $(K_n,p)$ in general position is an example of an abstract $d$-rigidity matroid, and this matroid may be distinct from
the generic $d$-dimensional rigidity matroid of $K_n$.

Although the axioms of abstract $d$-rigidity appear to be a rather specialised concept in rigidity theory, they actually define a very natural class of matroids on $E(K_n)$ if we consider some equivalent characterizations of abstract rigidity.  
To describe these, we first list some basic properties of a matroid $M$ on $E(K_n)$, which all hold when $M=\MR_d(K_n)$ and $n\geq d+1$.
\begin{description}
\item[(R1)] The rank of $M$ is $dn-{d+1\choose 2}$.
\item[(R2)] Every copy of $K_{d+2}$ is a circuit in $M$.
\item[(R3)] Every copy of $K_{1,n-d}$ is a cocircuit in $M$.
\item[(R4)] 
If $G\subseteq K_n$, $v$ is a vertex of $G$ of degree $d$ and $G-v$ is $M$-independent  then $G$ is $M$-independent.
\end{description}
\begin{theorem}[Graver-Servatius-Servatius\cite{GSS93} and Nguyen~\cite{N10}]\label{thm:hang}
Let $n,d$ be positive integers with $n\geq d+1$ and $M$ be a matroid on $E(K_n)$.
Then, the following statements are equivalent.
\begin{itemize}
    \item $M$ is an abstract $d$-rigidity matroid.
    \item $M$ satisfies (R1) and (R2).
    \item $M$ satisfies (R1) and (R3).
    \item $M$ satisfies (R1) and (R4).
    \item $M$ satisfies (R2) and (R3).
    \item $M$ satisfies (R2) and (R4).
\end{itemize}
\end{theorem}

Let us consider the case when $d=1$ as an example.
The graphic matroid of $K_n$ is the generic $1$-dimensional rigidity matroid, and hence it is an abstract $1$-rigidity matroid.
Indeed, the graphic matroid has rank $n-1$, implying (R1),
and every copy of $K_3$ is a circuit, implying (R2).
Conversely, suppose $M$ is an abstract 1-rigidity matroid of $K_n$.
Then, by (R4), any forest is $M$-independent, and by (R1), any cycle is $M$-dependent.
Hence, $M$ is the graphic matroid of $K_n$.
Thus, the graphic matroid of $K_n$ is the unique  abstract 1-rigidity matroid of $K_n$.
We will see that
the class of abstract $d$-rigidity matroids is much rich when $d\geq 2$.

Whiteley \cite{Wchapter} observed that  the generic $C^{d-2}_{d-1}$-cofactor matroid of $K_n$ introduced in Section~\ref{sec:3-rigidity} is another important example of an abstract $d$-rigidity matroid. 

\begin{prop}\label{prop:abstract1}
The generic $C^{d-2}_{d-1}$-cofactor matroid of $K_n$  is an  abstract $d$-rigidity matroid.
\end{prop}

Several basic properties of generic rigidity matroids are valid for any abstract rigidity matroid. 
For example, the following lemma shows the Maxwell's condition (Lemma~\ref{lem:max}) holds at the level of abstract rigidity, see \cite[Lemma 4.1]{CJT0}.

\begin{lemma}\label{lem:abstrac_max}
Let $M$ be a matroid on the edge set of $K_n$.
Suppose that $M$ satisfies (R2).
Then, for any $M$-independent subgraph $G=(V,E)$ of $K_n$, 
$i_G(X)\leq d|X|-{d+1\choose 2}$ for all $X\subseteq V$ with $|X|\geq d+1$ .
\end{lemma}

\subsubsection{The Graver-Whiteley Maximality Conjecture}\label{sec:max}
Graver's main motivation for defining abstract $d$-rigidity was to  gain a better understanding of the $d$-dimensional rigidity problem by characterising  the generic $d$-dimensional rigidity matroid as an extremal instance in the family of all abstract $d$-rigidity matroids.

For a family ${\cal M}$ of matroids on the same ground set,
the {\em weak order} $\preceq$ is defined as the partial order on $\MM$ such that, for any $M_1, M_2\in {\cal M}$,  $M_1\preceq M_2$ if every independent set in $M_1$ is independent in $M_2$.

Graver~\cite{G91} conjectured that the generic $d$-dimensional rigidity matroid of $K_n$ is the unique maximal matroid in the family of abstract $d$-rigidity matroids of $K_n$ under the weak order, and verified his conjecture when $d=2$.

\begin{theorem}\label{thm:abstract2}
The generic $2$-dimensional rigidity matroid of $K_n$ is the unique maximal matroid in the family of all abstract $2$-rigidity matroids of $K_n$ for sll $n\geq 3$.
\end{theorem}
Indeed, Theorem~\ref{thm:abstract2} is a direct consequence of Geiringer's theorem 
since every independent set in any abstract $2$-rigidity matroid satisfies Maxwell's condition by Lemma~\ref{lem:abstrac_max}
and such a set is independent in the generic $2$-dimensional-rigidity matroid by Theorem~\ref{thm:PG}.

Graver's conjecture turned out to be false for $d\geq 4$.
The first counterexample  was  given by N.~J.~Thurston (cf.~\cite{GSS93}) for $d=4$. 
Whiteley subsequently used the complete bipartite graph $K_{d+2,d+2}$ to show that the generic $C^{d-2}_{d-1}$-cofactor matroid provides a counterexample for all $d\geq 4$.
\begin{lemma}\label{lem:complete_bipartite} For all $d\geq 4$,
$E(K_{d+2,d+2})$ is a circuit in the generic $d$-dimensional rigidity matroid 
and  is independent in the generic $C^{d-2}_{d-1}$-cofactor matroid.
\end{lemma}
The rigidity part of Lemma~\ref{lem:complete_bipartite} follows from  the Bolker-Roth~\cite{BR} characterization of the rigidity of complete bipartite graphs.
The cofactor part follows from the fact that  the X-replacement operation preserves  independence  in the generic $C^{d-2}_{d-1}$-cofactor matroid~\cite{Wchapter}, cf. the discussion in the first paragraph of Section~\ref{sec:d4}.

Lemma \ref{lem:complete_bipartite} motivates the following modified conjecture.

\begin{conj}[Graver-Whiteley Maximality Conjecture]\label{conj:maximality}
The generic $C^{d-2}_{d-1}$-cofactor matroid of $K_n$ is the unique maximal matroid in the family of abstract $d$-rigidity matroids of $K_n$ for all $n\geq d+1\geq 2$.
\end{conj}

If true, this conjecture would give useful structural information about  the generic $C^{d-2}_{d-1}$-cofactor matroid. Indeed, it is shown in ~\cite{CJT} that the partial result that 
there exists a unique maximal abstract $d$-rigidity matroid $M$ of $K_n$ would imply that the following covering property holds for $M$:  for each cyclic flat $F$  of $M$, every $e\in F$ is contained in a copy of $K_{d+2}$ in $F$.

Clinch, Jackson, and Tanigawa~\cite{CJT0} verified  Conjecture \ref{conj:maximality} for $d=3$.
\begin{theorem}\label{thm:maximality3}
The generic $C^{1}_{2}$-cofactor matroid of $K_n$ is the unique maximal matroid in the family of abstract $3$-rigidity matroids of $K_n$ for all $n\geq 4$.
\end{theorem}
They then used the above mentioned covering property from~\cite{CJT} to prove Theorem \ref{thm:cjtproper}.
Despite these positive results,  Graver's original conjecture still remains open when $d=3$.
\begin{conj}[Graver's maximality conjecture for $d=3$]\label{conj:maximality3}
Let $n$ be a positive integer. Then the generic $3$-rigidity matroid of $K_n$ is the unique maximal matroid in the family of abstract $3$-rigidity matroids of $K_n$.
\end{conj}

The truth of Conjecture~\ref{conj:maximality3} would give a combinatorial characterisation of generic rigidity in $\R^3$ by 
Theorems~\ref{thm:cjt} and \ref{thm:maximality3}.

As pointed out by Crespo and Santos \cite{CrespoSantos2023}, when $d\geq 4$ we do not even know if $\MR_d(K_n)\prec C^{d-2}_{d-1}(K_n)$ or the two matroids are incomparable in the weak order.

\subsubsection{Duality, symmetric powers, and Mason's problem}
We will describe a connection between 
abstract rigidity and the symmetric powers of a matroid introduced by Lov{\'a}sz~\cite{L} and Mason~\cite{M}.

Let $K_V^{\circ}$ be the complete graph with vertex set $V$ and with a loop at each vertex and $K_{1,t}^{\circ}$ be the star with $t=1$ vertices and a loop at its central vertex.
We denote the loop at a vertex $v$ by the multiset $\{v,v\}$.

Given a matroid $M$ on a finite set $V$ of rank $t$, 
a matroid $N$ on $E(K_V^{\circ})$ is said to be a {\em second symmetric power} of $M$ if 
\begin{itemize}
\item[(S1)] the rank of $N$ is ${t+1\choose 2}$, and 
\item[(S2)] the rank in  $N$ of each  copy of $K_{1,t}$ and $K^{\circ}_{1,t}$ 
in $K_V^{\circ}$ is equal to the rank of the neighbours of its `central vertex' 
in $M$. More precisely, 
$r_M(X)=r_N(vX)$ for all $X\subseteq V$ and all $v\in V$, where 
$vX=\{\{v, x\}: x\in X\}$.
\end{itemize} 

A given matroid $M$ may have several or no  second symmetric powers, see, for example, \cite{LasVergnas1981} for more details.
 In the special case when  $M$ is linear, however, we can always construct a  canonical example of a symmetric second power by using
 the symmetric power of the ambient space of its linear representation.
Specifically, suppose $M$ has a linear representation over a field $\mathbb{F}$ by associating  a vector $p_v\in \mathbb{F}^t$ to each  $v\in V$ .
Then we can construct  a linear matroid on $E(K_n^{\circ})$ by associating the symmetric power $p_up_v$ with each $uv\in E(K_n^{\circ})$.\footnote{We will be concerned with the case when $\mathbb{F}=\R$. In this case, or any other case when $\mathbb{F}$ has characteristic zero, we can take the symmetric power to be the symmetric matrix  $p_up_v^T+p_vp_u^T$.} The resulting  matroid $N$ will satisfy (S1)  by the fact that the second symmetric power of $\mathbb{F}^t$ has dimension ${t+1\choose 2}$ and $N$ will also satisfy (S2) by bilinearity.

For example, the uniform matroid $U_n^t$ of rank $t$ with $n$ elements   can be  linearly represented by 
assigning a generic vector $p_v\in \mathbb{R}^t$ to each $v\in V$, when $|V|=n$.
Thus we can obtain a second symmetric power of $U_n^t$  by assigning $p_up_v$ to each $uv\in E(K_n^{\circ})$. We will refer to this matroid as the {\em generic second power of $U_{n,t}$} and refer to
its restriction to $E(K_n)$ (by removing the loops $ \{v,v\}$ of $K_n^{\circ})$ as
the {\em generic symmetric $t$-tensor matroid of $K_n$.}
A recent result of Brakensiek et al.~\cite{Brakensiek2024} tells us that the generic symmetric $t$-tensor matroid of $K_n$
is the dual matroid of the $d$-dimensional generic rigidity matroid of $K_n$ when we take $d=n-1-t$.

\begin{theorem}\label{thm:dual}
Let $n, d, t$ be positive integers with $d+t=n-1$.
Then the generic symmetric $t$-tensor matroid of $K_n$ is the dual of the generic $d$-dimensional rigidity matroid.
\end{theorem}

Motived by Theorem~\ref{thm:dual}, we can also consider the dual  of the other matroids in the family of abstract $d$-rigidity matroids.
We will see that this gives rise to other second symmetric powers of $U_n^t$ restricted to $E(K_n)$. 
Specifically, consider axiom (S2) in the case  when $M=U_n^t$ and we take the restriction  of $N$ to to $E(K_n)$.
Then  (S2) can be  simplified to: 
\begin{itemize}
\item[(S2')] Every copy of $K_{1,t+1}$ is a circuit in $N$.
\end{itemize}
We can now define the family of {\em abstract symmetric $t$-tensor matroids} to be the family of matroids on $E(K_n)$
satisfying (S1) and (S2').

For $d+t=n-1$, a matroid $N$ satisfies (S1) if and only if its dual $N^*$ satisfies (R1)
whereas $N$ satisfies (S2') if and only if $N^*$ satisfies (R3).
We can combine this observation with Theorem \ref{thm:hang} to obtain the following duality relation from~\cite{JacksonTanigawa2025}.

\begin{lemma}\label{lem:dual}
Let $n, d, t$ be positive integers with $d+t=n-1$. Then
a matroid $N$ on $E(K_n)$ is an abstract symmetric $t$-tensor matroid  if and only if its dual $N^*$ 
is an abstract $d$-rigidity matroid.
\end{lemma}

Lemma~\ref{lem:dual} enables us to make dual formulations of  the conjectures in Section \ref{sec:max}  in terms of symmetric $t$-tensor matroids. The dual formulation of Graver's maximality conjecture can be verified using  techniques from rigidity theory when  $t$ is small.

\begin{theorem}\label{thm:dual_graver}
Let $n, t$ be positive integers with $2\leq t\leq n-1$ and $t\leq 5$.
Then the generic symmetric $t$-tensor matroid of $K_n$  is the unique maximal element in the family of abstract symmetric $t$-tensor matroids of $K_n$.
Thus, Graver's maximality conjecture for generic rigidity is true whenever  $d\geq n-6$. 
\end{theorem}
 Theorem~\ref{thm:dual_graver} is proved in \cite{JacksonTanigawa2025} by providing a 
sparsity condition which characterises independence in  
 the generic symmetric $t$-tensor matroid when $t\leq 5$.
 It is also possible to derive Theorem~\ref{thm:dual_graver} from the characterisation of  all $\MR_d$-circuits on $n$ vertices when  $n\leq d+6$ given in \cite{GraseggerGulerJacksonNixon2022}.  

Mason~\cite{M} asked if there is always  a unique maximal element in the family of second symmetric powers of a matroid.
Las Vergnas~\cite{LasVergnas1981} showed that this property does not hold for  the closely related family of  {\em quasi second symmetric powers}, but  to the best of our knowledge Mason's original question is still  open. 
We formulate Mason's question for the special case of uniform matroids  as the following conjecture.

\begin{conj}\label{conj:maximality_symtensor}
Let $n, t$ be positive integers. Then the family of second symmetric powers of $U_n^t$ has  
a unique maximal element.
\end{conj}

In view of Theorem \ref{thm:dual_graver}, it is tempting to conjecture that the generic second symmetric power of $U_n^t$ will be the unique maximal element in Conjecture \ref{conj:maximality_symtensor}, but this is false when $t\geq 7$. 
We will use rigidity theory  to show this but 
we first need to fill a small gap between second symmetric powers of $U_{n,t}$, which are defined  on $E(K_n^{\circ})$, and abstract $d$-rigidity matroids, which are defined on $E(K_n)$.

Suppose $N$ is  a second symmetric power of $U_n^t$.
Then, its restriction to $E(K_n)$ always gives an  abstract symmetric $t$-tensor matroid
by definition.
We can use matroid duality to give an inverse construction.
Suppose $N$ is  an  abstract symmetric $t$-tensor matroid on $E(K_n)$ and let $N^*$ be its dual. Then $N^*$ is an abstract $d$-rigidity matroid with $d=n-t-1$ by Lemma \ref{lem:dual}.
We can extend $N^*$ to a matroid $\hat N^*$ on $E(K_n^{\circ})$ by appending each loop edge of $K_n^\circ$ as a loop element in the matroidal sense.
Let $B_n$ be the bicircular matroid on $E(K_n^{\circ})$
and let $\hat N^*\vee B_n$ be the matroid union of $\hat N^*$ and $B_n$.
Then  $\hat N^*\vee B_n$  has rank $(d+1)n-{d+1\choose 2}$ and 
every copy of $K_{1,n-d+1}$ and $K_{1,n-d}^{\circ}$ is a cocircuit.
Therefore, its dual  $(\hat N^*\vee B_n)^*$ satisfies (S1) and (S2) for $U_n^t$,
and hence is a second symmetric power of $U_n^t$.

We can now apply the above construction taking $N$ to be the dual of the generic $C_{d-1}^{d-2}$-cofactor matroid of $K_n$. Then 
$M=(\hat N^*\vee B_n)^*$  is a second symmetric power of $U_n^t$.
In addition, we have $M\not \preceq M'$ when, $M'$ is the generic second symmetric power of $U_n^t$, $t\geq 7$ and $n\geq 2t-2$.  This follows from Theorem~\ref{thm:dual} and Lemma~\ref{lem:complete_bipartite} by taking the dual of the restrictions of $M$ and $M'$ to $E(K_n)$. We can now deduce:

\begin{theorem}
The generic second symmetric power of $U_n^t$ is not the the unique maximal element in the family of all second symmetric powers of $U_n^t$ whenever $t\geq 7$ and $n\geq 2t-2$.
\end{theorem}

\subsubsection{Birigidity}
As a bipartite analogue of rigidity matroids, Kalai, Nevo, and Novik~\cite{KNN} introduced birigidity matroids.
Let $G=(V_1\cup V_2, E)$ be a bipartite graph, and consider point-configurations $p_1:V_1\rightarrow \mathbb{R}^{d_2}$ and  $p_2:V_1\rightarrow \mathbb{R}^{d_1}$.
The {\em birigidity matrix} of $(G,p_1,p_2)$ is defined as the $|E|\times (d_1+d_2)|V|$ matrix in which the row associated with an edge $e=uv$ is given by 
\[
\kbordermatrix{
  & &  u & & v & \\
 e=uv & 0 \dots 0 & p_v & 0\dots 0 & p_u & 0\dots 0
}.
\]
The row matroid of the birigidity matrix defines a matroid on $E$, which is called the {\em birigidity matroid} of $(G,p_1,p_2)$.
The {\em generic $(d_1,d_2)$-birigidity matroid} of $G$, $\MR_{d_1,d_2}(G)$,  is the birigidity matroid of $(G,p_1,p_2)$ for generic $p_1, p_2$.
The special case when $d_1=d_2$ has been  studied in the context of the low matrix completion problem~\cite{SC}.
The case when $d_1=1$ or $d_2=1$ coincides with Whiteley's matroid from scene analysis~\cite{Whiteley1989}.

A natural analogue of the gluing property holds in birigidity, and hence we can develop a bipartite analogue of abstract rigidity and formulate Graver-type questions in this bipartite setting.
In the subsequent discussion, we will assume that all bipartite graphs $G=(V_1\cup V_2,E)$ are given with a fixed bipartition $(V_1,V_2)$ and will refer to $V_1$ and $V_2$ as the {\em left}  vertex set and the {\em right} vertex set of $G$, respectively.

Suppose $n_i, d_i$ are integers with $n_i\geq d_i\geq 1$ for $i=1,2$ and $M$ is a matroid on  $E(K_{n_1,n_2})$. Then we say that $M$ is an
{\em abstract $(d_1,d_2)$-birigidity matroid} if its closure operator $\cl_M$ satisfies the following two axioms.
\begin{description}
\item[(BG1)] Suppose $E_i\subseteq V_{1,i}\times V_{2,i}\subseteq E(K_{n_1,n_2})$   for $i=1,2$, $|V_{1,1}\cap V_{1,2}|< d_2$ and  $|V_{2,1}\cap V_{2,2}|< d_1$.
Then ${\rm cl}_M(E_1\cup E_2)\subseteq  (V_{1,1}\times V_{2,1})\cup (V_{1,2}\times V_{2,2})$.
\item[(BG2)] Soppse $E_i\subseteq V_{1,i}\times V_{2,i}\subseteq E(K_{n_1,n_2})$  with ${\rm cl}_M(E_i)=V_{1,i}\times V_{2,i}$ for $i=1,2$, $|V_{1,1}\cap V_{1,2}|\geq d_2$ and $|V_{2,1}\cap V_{2,2}|\geq d_1$. Then 
${\rm cl}_M(E_1\cup E_2)= (V_{1,1}\cup V_{1,2})\times (V_{2,1}\cup V_{2,2})$.
\end{description}
The generic $(d_1,d_2)$-birigidity matroid is an example of an abstract $(d_1,d_2)$-birigidity matroid.

We can give equivalent characterisations of abstract birigidity matroids using bipartite analogues of properties (R1)-(R4) in Section~\ref{sec:abstract}. Consider the following properties of a matroid $M$ on $E(K_{n_1,n_2})$, which all hold when $M=\MR_{d_1,d_2}(K_{n_1,n_2})$ and  $n_i\geq d_i\geq 1$ for $i=1,2$.
\begin{description}
\item[(BR1)] The rank of $M$ is $d_1n_1+d_2n_2-d_1d_2$.
\item[(BR2)] Every copy of $K_{d_2+1,d_1+1}$ is a circuit in $M$.
\item[(BR3)] Every copy of $K_{n_1-d_2+1,1}$ and $K_{1,n_2-d_1+1}$ is a cocircuit in $M$.
\item[(BR4)] Suppose $G\subseteq K_{n_1,n_2}$ and $v$ is either a left vertex of $G$ of degree $d_1$ or a right vertex of $G$ of degree $d_2$. If $G-v$ is independent in $M$, then $G$ is $M$-independent.
\end{description}

\begin{theorem}\label{thm:abstract_birigidity}
Let $n_1,n_2,d_1,d_2$ be positive integers with $n_i\geq d_i$ and $M$ be a matroid on $E(K_{n_1,n_2})$.
Then, the following statements are equivalent.
\begin{itemize}
\item $M$ is an abstract $(d_1,d_2)$-birigidity matroid.
\item $M$ satisfies (BR1) and (BR2).
\item $M$ satisfies (BR1) and (BR3).
\item $M$ satisfies (BR1) and (BR4).
\item $M$ satisfies (BR2) and (BR3).
\item $M$ satisfies (BR2) and (BR4).
\end{itemize}
\end{theorem}

The bipartite counterpart of Graver's conjecture is the following.
\begin{conj}[Birigidity maximality conjecture]\label{conj:maximality_bi}
Let $n_1,n_2, d_1, d_2$ be positive integers. Then the generic $(d_1,d_2)$-rigidity matroid of $K_{n_1,n_2}$ is the unique maximal matroid in the family of abstract $(d_1,d_2)$-birigidity matroids of $K_{n_1,n_2}$.
\end{conj}
Unlike Graver's original conjecture for abstract rigidity, there are no known counterexamples to Conjecture~\ref{conj:maximality_bi}.
Currently, the only resolved cases are those where $d_1 = 1$ or $d_2 = 1$, in which case Whiteley's theorem~\cite{Whiteley1989} on scene analysis implies Conjecture~\ref{conj:maximality_bi}, and those where $d_1 \geq n_1 - 3$ or $d_2 \geq n_2 - 3$, which follow from the characterisation of the dual matroids due to Brakensiek et al~\cite{Brakensiek2024}.

The case when $d_1=d_2=2$ is the smallest  unsolved case.
Bernstein~\cite{Bernstein2017} gave a noteworthy combinatorial characterization of independece in the generic $(2,2)$-birigidity matroid.
This characterization does not provide a co-NP certificate for independence, however, and giving a good characterization (in the sense of computational complexity) is wide open.

As for generic rigidity, generic birigidity has a natural dual counterpart.
Brakensiek et al.~\cite{Brakensiek2024} showed that the dual of the generic $(d_1,d_2)$-birigidity matroid of $K_{n_1,n_2}$ is a  tensor product of $U_{n_1,n_1-d_1}$ and $U_{n_2,n_2-d_2}$.\footnote{The family of tensor products of two (or more) matroids can be defined in a similar way to the family of symmetric powers of a single matroid, see \cite{M, LasVergnas1981}.}
The following dual counterpart to f Conjecture~\ref{conj:maximality_bi} is {a special case of another question of
Mason~\cite{M} for tensor products of matroids}.
\begin{conj}\label{con:tensor}
    There is a unique maximal element in the family of tensor products of two uniform matroids.
    Moreover, this unique maximal element is the dual of the generic birigidity matroid. 
\end{conj}
A result by Brakensiek et al.~\cite{Brakensiek2024} verifies Conjecture~\ref{con:tensor} in the case when one matroid is a uniform matroid of rank three and the other is a uniform matroid of arbitrary rank.
See also B\'erczi et al.~\cite{Berczi} for further recent progress.


\printbibliography

\myaddress

\end{document}